\documentclass[11pt,leqno]{amsart}
\usepackage{amsmath,amssymb,txfonts}
\usepackage{amssymb}
\usepackage{amsxtra}
\usepackage{amsmath}
\usepackage{txfonts}

\usepackage{tikz}

 \usepackage{color}

\usepackage[cp1252]{inputenc}
\usepackage{mathrsfs}
\usepackage{graphicx}
\usepackage[active]{srcltx}

\allowdisplaybreaks

\def\rr{{\mathbb R}}
\def\rn{{{\rr}^n}}

\def\nn{{\mathbb N}}

\def\fz{\infty}
\def\az{\alpha}
\def\dist{{\mathop\mathrm{\,dist\,}}}
\def\loc{{\mathop\mathrm{\,loc\,}}}

\def\lz{\lambda}
\def\dz{\delta}

\def\ez{\epsilon}

\def\bz{\beta}

\def\gz{{\gamma}}

\def\wz{\widetilde}

\def\esup{\mathop\mathrm{\,esssup\,}}

\newtheorem{thm}{Theorem}[section]
\newtheorem{lem}[thm]{Lemma}
\newtheorem{prop}[thm]{Proposition}
\newtheorem{rem}[thm]{Remark}
\newtheorem{cor}[thm]{Corollary}

\newtheorem{defn}[thm]{Definition}

\numberwithin{equation}{section}

\begin{document}
\arraycolsep=1pt

\title[Absolute minimizers  for convex Hamiltonians]
{Regularity of absolute minimizers for continuous convex Hamiltonians}

\author{ Peng Fa, Changyou Wang, and Yuan Zhou}

          \address{ Department of Mathematics, Beihang University, Beijing 100191, P.R. China}
                    \email{SY1609131@buaa.edu.cn }

                      \address{ Department of Mathematics, Purdue University,
 West Lafayette, IN 47907, USA}
                    \email{wang2482@purdue.edu}

\address{ Department of Mathematics, Beihang University, Beijing 100191, P.R. China}
                    \email{yuanzhou@buaa.edu.cn}

\thanks{    }

\date{\today }

 \begin{abstract}
 {For any $n\ge 2$, $\Omega\subset\rn$,  and any given convex and coercive Hamiltonian function $H\in C^{0}(\rn)$,
we find an optimal sufficient condition on $H$, that is,  for any $c\in\mathbb R$, the level set $H^{-1}(c)$ does not contains any line segment, such
then any absolute minimizer $u\in AM_H(\Omega)$ enjoys the linear approximation property.
As consequences, we show that when $n=2$, if $u\in AM_H(\Omega)$ then $u\in C^1$; and  if $u\in AM_H(\rr^2)$
satisfies a linear growth at the infinity, then $u$ is a linear function on $\rr^2$.  
In particular, if $H$ is a strictly convex Banach norm $\|\cdot\|$ on $\mathbb R^2$,
e.g. the $l_\alpha$-norm  for $1<\alpha<1$, then any $u\in AM_H(\Omega)$ is $C^1$.
The ideas of proof are, instead of PDE approaches,  purely variational and geometric.

}
\end{abstract}
\maketitle
\tableofcontents
\section{Introduction}

For $ n\ge2$, assume that $H\in C^0(\rn)$ is a continuous function that
is convex, and coercive, i.e.,
$$\displaystyle\lim_{|p|\to\fz}H(p) =\fz.$$
In a series of papers \cite{a1,a2,a3,A1968,a4},
G. Aronsson initiated the study of  minimization problems involving
the supremum norm (or $L^\infty$) functional:
$${\mathcal F}_{H}\left(u,\Omega\right)=\esup_{x\in\Omega}H\left(Du (x) \right),$$
$\mbox{for any domain $\Omega\subset\rn$ and function $u\in W^{1,\fz}\left(\Omega\right)$}.$

According to Aronsson \cite{a1,a2}, given a domain $\Omega\subset\rn$
a function $u\in W^{1,\infty}_{\rm{loc}}\left(\Omega\right)$ is called an absolute minimizer  of $H$,
or $u\in AM_H\left(\Omega\right)$ for brevity,  if
$${\mathcal F}_{H}\left(u,V\right)\le {\mathcal F}_{H}\left(v,V\right),$$  $\mbox{whenever $V\Subset\Omega$, $v\in W^{1,\infty}_{\loc}\left(V\right)\cap C (\overline V )$ satisfies $v=u$ on $\partial V$}.$

According to Crandall-Evans \cite{ce},   a function $u\in C^0\left(\Omega\right)$ is said to enjoy the linear approximation property, if for  any $x\in \Omega$ and any sequence $\{r_i\}_{i\in\nn}\rightarrow 0$,
there exist a subsequence $\{r_{i_k}\}_{k\in\nn}$
and a vector $e\in \rn$
such that
\begin{equation}\label{LAP}\lim_{k\to\fz}\sup_{y\in B\left(0,1\right)}\big|\frac{u (x+r_{i_k}y )-u\left(x\right)}{r_{i_k}} -  e\cdot y \big|=0.
\end{equation}
We let $\mathscr D u(x)$ denote the collection of all possible vector $e\in\rn$
appearing in \eqref{LAP}. It is readily seen that $u$ is differentiable at $x$ if and only if  $\mathscr D u(x)=\{D u(x)\}$ is a singleton.

A Lipschitz function does not necessarily satisfy the linear approximation property, for example if $u(x)=|x|$ then $\mathscr Du(0)=\emptyset$.
The following example,  due to D. Preiss,
 $$w(x_1,x_1')=\left\{\begin{array}{ll}  x_1 \sin\Big(\log \big| \log |x_1|\big|\Big),\quad&
   0<|x_1|<1,\ x_1'\in\rr^{n-1},\\
   0&  x_1=0,\ x_1'\in\rr^{n-1}, \end{array}\right.
   $$
indicates that a Lipschitz function, satisfying the linear approximation property \eqref{LAP}, may not be differentiable.

It was first shown by \cite{ce} that if $u\in W^{1,\infty}_{\rm{loc}}(\Omega)$ is an absolutely minimal Lipschitz extension (AMLE),
or equivalently an absolute minimizer of $H(p)=|p|^2$, then it
satisfies the linear approximation property \eqref{LAP} for any $x\in\Omega$, and
\begin{equation}\label{norm-grad}
|e|=S^+u(x)=\lim_{r\rightarrow 0} \frac{1}{r}\big(\max_{B_r(x)}u -u(x)\big), \ \forall e\in\mathscr Du(x).
\end{equation}
Wang-Yu \cite{wy} estabilshed the linear approximation property \eqref{LAP} for
$u\in AM_H(\Omega)$ at any $x\in\Omega$, and
\begin{equation}\label{slope}
H(e)=\lim_{r\rightarrow 0} \big\|H(D u)\big\|_{L^\infty(B_r(x))}, \ \forall e\in\mathscr Du(x),
\end{equation}
for any nonnegative, uniformly convex, and coercive $H\in C^2(\rn)$.

The linear approximation property \eqref{LAP} and \eqref{norm-grad} of an AMLE, or an infinity harmonic function,
$u\in C^{0,1}(\Omega)$, has played an important role in establishing
its $C^{1,\alpha}$-regularity by Savin \cite{s05} and Evans-Savin \cite{es08}
in dimension $n=2$, and its differentiability by Evans-Smart \cite{es11a,es11b}
in dimensions $n\ge 3$, see also \cite{wy}. \eqref{LAP} and \eqref{slope} has also played an important role for the $C^1$-regularity of absolute minimizers $u\in AM_H(\Omega)$ for $C^2$-uniformly
convex $H$ in dimension $n=2$ by Wang-Yu \cite{wy}.

The notion of absolute minimizers can be defined for any continuous
Hamiltonian $H$. It has been an outstanding open question whether the
the linear approximation property holds for $u\in AM_H(\Omega)$, if
we weaken the assumption that $H\in C^2(\rn)$ is uniformly convex to the natural condition
that $H\in C^0(\rn)$ is convex.

The main contribution of this paper gives an affirmative
answer to this problem by showing the linear approximation property
of absolute minimizers, provided $H\in C^0(\rn)$ is a convex function
whose level set does not contain any line segment. As consequences,
we are able to establish both the $C^1$-regularity and a Liouville property of absolute minimizers
for any such $H$ in dimension two.  More precisely, we have

\begin{thm}\label{lla}  For $n\ge 2$, if $H\in C^{0}(\rn)$ is convex and coercive, and satisfies

\noindent {\rm ({\bf A})} the level set $H^{-1}(c)$ does not contain any line segment
for any $c\in\mathbb R$,

\noindent then

\noindent {\rm ({\bf B})} any $u\in AM_H(\Omega)$, $\Omega\subset\rn$,
satisfies the linear approximation property \eqref{LAP} and \eqref{slope}.
\end{thm}

As an immediate application, we obtain both $C^1$-regularity and a Liouville property of absolute minimizers of $H\in C^0(\mathbb R^n)$ satisfying the condition ({\bf A})
of Theorem \ref{lla} for $n=2$.  More precisely, we have

\begin{thm}\label{lla0}
Assume $H\in C^{0}(\mathbb R^2)$ is convex  and coercive, and satisfies
condition ({\bf A}) of Theorem \ref{lla}. Then

\noindent{\rm({\bf C})} $u\in AM_H(\Omega)$, $\Omega\subset\rr^2$, is in $C^1(\Omega)$; and

\noindent{\rm({\bf D})}  if $\Omega=\rr^2$ and $u\in AM_H(\rr^2)$  has a linear growth at the infinity,
then $u$ is a  linear function in $\rr^2$.
\end{thm}
Recall a function $u\in C^0(\rn)$ has a linear growth
at the infinity, if  there exists a constant $C>0$ such that
\begin{equation} \label{LG}
 |u(x)|\le C(1+|x|), \ \forall x\in\rn.
 \end{equation}
({\bf D}) of Theorem \ref{lla0} is usually referred  as a Liouville property.

 {We remark that for all $n\ge 2$, the condition ({\bf A}) is optimal (necessary in some sense) for the properties ({\bf B}), ({\bf C}) or  ({\bf D})
as in Theorems \ref{lla} and \ref{lla0}. In fact, if ({\bf A}) of Theorem \ref{lla} were false, then there would
exist a $c\in \mathbb R$ and a line segment $[a, b]\subset \rn$, with $a\not=b$,
such that $H(y)=c$ for any $y\in [a,b]$. Then we can modify an example by
Katzourakis \cite{k11} to construct an $u\in AM_H(\Omega)$, which satisfies
none of the properties ({\bf B}), ({\bf C}), and ({\bf D}) stated in Theorems \ref{lla} and \ref{lla0},
see Section 4 for details.}

Let $(\rn, \|\cdot\|)$ be a Banach space and $H(p)=\|p\|$, for $p\in\rn$,
be the Banach norm.
Then it is not hard to see that the following three statements are equivalent:\\
(1) $H^{-1}(c)$ does not contain line segments for any $c\ge 0$.\\
(2) The unit sphere $\big\{p\in\rn: \|p\|=1\big\}\subset\rn$ does not contain any line-segment.\\
(3) $H(p)=\|p\|$ is strictly convex.\\
In particular, if we consider the $l_\alpha$-norm on
$\rn$ ($n\ge 2$):
\begin{equation}\label{alpha}
|p|_{\alpha}
=\begin{cases} \displaystyle\Big(\sum_{i=1}^n|p_i|^\alpha \Big)^{\frac1{\alpha}},
& 1\le \alpha<\infty,\\
\displaystyle\max_{i=1}^n |p_i|, & \alpha=\infty,
\end{cases}
\end{equation}
and define $H(p)=|p|_{\alpha}$, then $H$ satisfies the condition
({\bf A}) of Theorem \ref{lla} if and only if $1<\alpha<\infty$.

As an immediate consequence of Theorem \ref{lla} and Theorem \ref{lla0},
we have
\begin{cor} \label{lla1} For $n\ge 2$, if a Banach norm $\|\cdot\|$ in $\rr^n$
is strictly convex, then
any absolute minimizer $u\in AM_H(\Omega)$ for $H(p)=\|p\|$ satisfies the linear approximation property \eqref{LAP} and \eqref{slope}.  For $n=2$, we have that

\begin{itemize}
\item [a)] any absolute minimizer $u\in AM_H(\Omega)$ of $H(p)=\|p\|$ is
in $C^1(\Omega)$, and
\item [b)] any absolute minimizer $u\in AM_H(\mathbb R^2)$ of $H(p)=\|p\|$,
satisfying the linear growth condition \eqref{LG}, must be a linear function.
\end{itemize}
\end{cor}
\begin{cor}\label{lla2} For $1<\alpha<\infty$, $\Omega\subset\mathbb R^2$,
if $H(p)=|p|_\alpha$ is the $l_\alpha$-norm, then  $u\in AM_H(\Omega)$ is in $C^1(\Omega)$.
If, in addition, $\Omega=\mathbb R^2$ and $u$ satisfies the linear growth property
\eqref{LG}, then $u$ is a linear function.
 \end{cor}

When $H\in C^1(\rn)$,   Aronsson \cite{a1,a2,a4}
derived the Euler-Lagrange equation, called as Aronsson's equation nowdays,
for an absolute minimizer $u\in AM_H(\Omega)$:
 \begin{equation}\label{eq1.2}
\mathscr A_H[u]:= \sum_{i,j=1}^n
H_{p_i} \left(Du\right) H_{p_j} \left(Du\right) u_{x_ix_j}     =0\quad\mbox{\rm in}\;\Omega.
\end{equation}
When  $H(p)= \frac12|p|^2$ for $p\in\rn$,
 \eqref{eq1.2} reduces to  the  $\infty$-Laplace equation
 \begin{equation}\label{eq1.3}\Delta_\infty u:=\sum_{i,j=1}^n
u_{x_i}   u_{x_j}   u_{x_ix_j} =0\quad{\rm in }\ \Omega.
\end{equation}

It is well-known that Crandall-Lions' viscosity solution
theory \cite{CIL} can be employed to study both \eqref{eq1.2} and \eqref{eq1.3}.
Jensen  \cite{j93} was the first to show that
an  AMLE is equivalent
to a viscosity solution to \eqref{eq1.3} (or an infinity
harmonic function), both of which are  unique under the Dirichlet
boundary condition $u|_{\partial\Omega}=g\in C(\partial\Omega)$,
see also \cite{as, bb,cgw,pssw} for alternate approaches to the uniqueness.
In general, when $H\in C^1(\rn)$ is convex and coercive,
through Crandall-Wang-Yu \cite{cwy}   and  Yu  \cite{y06} we know that
a viscosity solutions to   Aronsson's equation \eqref{eq1.2}
is equivalent to an absolute minimizer for $H$, see also \cite{acjs, bjw,c03,gwy06}.
Barron-Jensen-Wang \cite{bjw}
have obtained an existence result of absolute minimizers
for general $H(x,z,p)\in C(\Omega\times \mathbb R\times\rn)$, that
is level-set convex in the $p$-variable, which is a viscosity solution
of Aronsson's equation under some further assumptions
on $H$, see also \cite{c03}. The uniqueness of absolute minimizers was subsequently proved by Jensen-Wang-Yu \cite{bjw} and Armstrong-Crandall-Julin-Smart \cite{acjs}
for convex and coercive $H\in C^2(\rn)$  and $H\in C^{0}(\rn)$ respectively,
provided the minimal level set $\displaystyle H^{-1}(\displaystyle\min_{\rn}H)=\big\{p\in\rn: H(p)=\min_{\rn} H \big\}$ has an empty interior.

It is readily seen that the condition ({\bf A}) of Theorem \ref{lla} implies
that the minimal level set of $H$ has en empty interior, hence the uniqueness
holds for absolute minimizer of $H$ satisfying ({\bf A}) of Theorem \ref{lla}.
On the other hand, it is easy to construct a convex and coercive
$H\in C^0(\rn)$ such that $\big(H^{-1}({\rm{min}}_{\rn}H)\big)$ has an
empty interior, but $H^{-1}(c)$ contains a line segment for some $c\in\mathbb R$. According
to Theorem \ref{lla}, for any such a $H$ there exists an absolute minimizer $u$
that does not enjoy  \eqref{LAP} and \eqref{slope}. This indicates that the linear approximation property \eqref{LAP} and \eqref{slope} for absolute minimizers
is, in fact, stronger  than the property of uniqueness  under the Dirichlet
boundary condition.

A few remarks on our main results are in order:
\begin{rem} {\rm a) Theorem \ref{lla} and Theorem \ref{lla0} have previously
been shown by Wang-Yu \cite{wy} where $H$ is $C^2(\rn)$ and uniformly convex. Our primary advances assert that both Theorems remain to be true under the natural condition that $H\in C^0(\rn)$ is convex,
and satisfies condition ({\bf A}) of Theorem \ref{lla}. It is worthwhile pointing out that
the uniform convexity of $H$ implies condition
({\bf A}) of Theorem \ref{lla}. \\
b) If $H$ is a uniformly convex, $C^1$-function in $\rn$, then Theorem
\ref{lla} and Theorem \ref{lla0} also hold for viscosity solutions of
\eqref{eq1.2}, since an absolute minimizer of $H$ is equivalent
to a viscosity solution of \eqref{eq1.2} (see \cite{cwy} and \cite{y06}).
However, if $H$ is merely $C^0$ in $\rn$, it is unknown whether
Theorem \ref{lla} and Theorem \ref{lla0} hold for viscosity solutions
of \eqref{eq1.2}, which can be defined
by replacing $H_p(Du)$ by $q\in \partial H(Du)$ (the subdifferential of
$H$), see \cite{acjs} or c) below. It remains to be an open
question that a viscosity solution of \eqref{eq1.2} is an absolute minimizer of $H$ in this class.\\
c) For $n\ge 2$, if $H_\alpha(p)=|p|_\alpha, p\in\rn,$ is the $l_\alpha$ for $1\le\alpha\le\infty$, it follows from \cite{cgw} that an absolute minimizer $u\in C^{0,1}(\Omega)$ of $H_\alpha$ if and only if $u$ is an infinity harmonic function with respect to $l_\alpha$-norm, i.e.,
\begin{equation}\label{eq1.3'}
\begin{cases}
\forall (x, \phi) \in \Omega\times C^2(\Omega), x\ \mbox{is a local maximum of}\ u-\phi
\Longrightarrow \displaystyle\max_{q\in\partial H_\alpha(D\phi(x))}\langle D^2\phi(x)q,q\rangle\ge 0,\\
\forall (x, \phi)\in (\Omega, C^2(\Omega)), x\ \mbox{is a local minimum of}\ u-\phi
\Longrightarrow \displaystyle\min_{q\in\partial H_\alpha(D\phi(x))}\langle D^2\phi(x)q,q\rangle\le 0.
\end{cases}
\end{equation}
Corollary \ref{lla1} and Corollary \ref{lla2} imply that for any $\alpha\in (1,\infty)$,
an infinity harmonic function $u$, with the $l_\alpha$-norm, enjoys the
linearity approximation property \eqref{LAP} and \eqref{slope} for $n\ge 2$;
and is $C^1$ and enjoys the Liouville property in $\mathbb R^2$. It would be
interesting to ask whether an infinity harmonic function with the
$|\cdot|_\alpha$-norm is $C^{1,\alpha}$ in $\mathbb R^2$, and differentiable in $\rn$ for $n\ge 3$.}
\end{rem}

\subsection{Outline of ideas of proofs}

First, observe that in Theorems \ref{lla} and \ref{lla0}, we may assume
$H$ satisfies the stronger condition:
\begin{equation}\label{assh0}\mbox{$H\in C^0(\rn)$ is convex, superlinear, that is $\displaystyle\lim_{|p|\to\fz}   \frac{H\left(p\right)}{|p|} =\fz$,
and $\displaystyle H\left(0\right)=\min_{p\in\rn} H\left(p\right)=0$.}
\end{equation}
Indeed, if  $H\in C^0(\rn)$ is convex  and coercive, then
 there exists a $p_0\in\rn$ such that $\displaystyle H(p_0)=\min_{p\in\rn}H(p)$.
Set $\wz H(p)=\big(H(p+p_0)-H(p_0)\big)^2$ for $p\in\rn$.
Then it is easy to show\\
\indent 1) $\wz H$ satisfies \eqref{assh0};\\
\indent 2) $\wz H$ satisfies ({\bf A}) of Theorem \ref{lla} if and only if the same holds for $H$;\\
\indent 3) $u\in AM_{H} (\Omega ) $  if and only if $\wz u(x) \in AM_{\wz H} (\Omega )$,
where  $\wz u(x)=u(x)-p_0\cdot x$ for all $x\in\Omega$; and\\
\indent 4) $u$ satisfies ({\bf B}), ({\bf C}), ({\bf D}) of Theorems \ref{lla} and \ref{lla0}
if and only if the same holds for $\wz u$.\\
Thus, from now on, we will assume that $H$ satisfies \eqref{assh0}.

Let $L=H^*$ be the Legendre transform (or the convex conjugate) of $H$:
\begin{equation}\label{conj}L\left(q\right)=\sup_{p\in\rn}\big(p\cdot q -H\left(p\right)\big).
\end{equation}
Then $L$ satisfies \eqref{assh0} and $H=L^*$ is also the convex conjugate of $L$. (i) of Theorem \ref{lla} guarantees that \cite{acjs} is applicable, and
hence any $u\in AM_H(\Omega)$ enjoys the convexity or concavity criteria,
the comparison principle, and the property of comparison with cones,
see Section 2  for details.


The proof of  ({\bf A}) $\Rightarrow$ ({\bf B}) of Theorem \ref{lla} relies on  the following property, see Section 5 below for details.

\medskip
\noindent{\bf Theorem 5.1}. {\it For $n\ge 2$, assume $H$ satisfies both \eqref{assh0}
and condition ({\bf A}) of Theorem \ref{lla}.
If $u\in C^{0,1}(\rn)$   satisfies
\begin{equation*}
\mbox{$S^+_tu(0)=-S^-_tu(0)=k$,
and $\max\big\{S^+_tu(x), -S^-_tu(x)\big\}\le k$\ $\forall x\in\rn$, $t>0$   }
  \end{equation*}
  for some $0\le k<\fz$, then
  there exists a vector $p_0\in\rn$ such that $H(p_0)=k$ and
  $$\mbox{$u\left(x\right)=u\left(0\right)+p_0\cdot x$,\ $\forall   x\in\rn$.}$$
}
\indent Theorem 5.1 was first proven by Crandall-Evans \cite{ce}
for $H(p)=\frac12|p|^2$ by using the Hilbert structure, later by
Wang-Yu \cite{wy} for uniformly convex $H\in C^2(\rn)$
by using the $C^1$-regularity of cone functions $C^H_k(\cdot)$.
If $H\in C^1(\rn)$ is strictly convex and satisfies \eqref{assh0},
by using both  $C^1(\rn)$-regularity of $L$
and strict convexity of  $L $ (see Proposition \ref{LEMstrcon}  below),
we can deduce  Theorem 5.1  by adapting the arguments of Yu \cite{y07}.
However, if $H$ is merely continuous and satisfies \eqref{assh0} and
({\bf A}) of Theorem \ref{lla}, then it is not necessarily true that  $L$ is  either $C^1$  or strictly convex, see Proposition 3.2; and
it is also an open question that
the cone function $C^H_k(\cdot)\in C^1(\rn)$.
Thus, in order to show Theorem 5.1, we need to develop the following
new ideas:
\begin{enumerate}
\item[(a)] The subdifferential set of  $L$ at any $q\in\rn$ must be either
a singleton or a line segment on which $H$ is strictly monotone,
see Proposition \ref{LEMC3}.

\item[(b)] Based on the geometric property (a),
and some careful analysis on the analytic and geometric structure of
Hamilton-Jacobi flows and the subdifferential set of $L$,
either there exists a point $y^+\in\rn$ such that $u$ is linear in $\rr y^+$,
or there are a pair of points $y^\pm\in\rn$ such that  $u$ is linear in
$[sy^-,sy^+]$ for all $s\ge0$. This, with the help of this geometric structure of $\partial L$,
implies that there is a unique $p_0\in  \partial L(y^+)$
such that $H(p_0)=k$ and $u\left(x\right)=u\left(0\right)+p_0\cdot x,\forall x\in\rn,$
see Section 5 for details.
\end{enumerate}

For $n=2$, $\Omega\subset\mathbb R^2$, and
a function $u\in C^{0 }(\Omega)$,
let $x^0\in\Omega$, $\dz\in(0,1]$, and
$0<r<\min\{1, d(x^0,\partial\Omega)\}$, and denote by
$\mathscr D(u)(x^0;r;\dz)$   the set of
vectors $e\in\mathbb R^2 $ such that
\begin{equation}\label{delta-lap}
\sup_{B(x^0,r )}|u(x )-u(x^0)-e \cdot (x-x_0) | \le \dz r.
\end{equation}
In other words, $\mathscr D(u)(x^0;r;\dz)$ collects all the slopes of
linear approximations of $u$ in $B(x^0, r)$ at the scale $\dz$.

The proof of ({\bf A}) $\Rightarrow$ ({\bf C}) and ({\bf D}) in Theorem \ref{lla0}
is based on Theorem 6.1.

\medskip
\noindent{\bf Theorem 6.1}. {\it For $n=2$,
assume $H$ satisfies both \eqref{assh0}  and ({\bf A}) of Theorem \ref{lla}.
For each  $\ez >0$ and each vector  $e_8\in H^{-1}([1,2])  $
there exist  $ \dz_\ast(H,\ez,e_8 ) >0$  such that for
any $u \in AM_H(B(0,8))$ and any $0<\delta< \dz_\ast(H,\ez,e_8 )$,
we have
 \begin{equation*}
 \max_{e\in \mathscr D u(0)}| e_8 -e |\le \ez\quad
 \mbox{whenever
  $e_8 \in \mathscr D(u )(0;8;\frac{\dz}8)$.}
\end{equation*}}

\medskip
From Theorem \ref{lla} and Theorem 6.1, we can show
Theorem \ref{c2}, from which
({\bf C}) and   ({\bf D})  of Theorem \ref{lla0} follow in a rather standard routine,
see Section 7 for details.

\medskip
\noindent{\bf Theorem 7.1}. {\it For $n=2$, assume
 $H$ satisfies both \eqref{assh0} and ({\bf A}) of Theorem \ref{lla}.
 \begin{enumerate}
 \item
For any domain $\Omega\subset\mathbb R^2$,
if $u\in AM_H\left(\Omega\right)$ then $u\in C^1\left(\Omega\right)$.
\item For any $k>0$, there exists a continuous, monotone increasing
function $\rho_k$,   with $\rho_k\left(0\right)=0$, such that for any
$z\in\mathbb R^2$ and $r>0$,
if $v\in AM_H(B(z,2r))$ satisfies $\|H(Dv)\|_{L^\infty(B(z,2r))}\le k$, then
 \begin{equation*}
 \sup_{ x,y\in B\left(z,s\right)}
  |Dv(x)-Dv(y)|\le \rho_k\big(\frac{s}{r}\big), \ \forall 0<s<r.
  \end{equation*}
  \end{enumerate}
}

\medskip
Recall that  a stronger version of
Theorem \ref{xprop}, with
$\delta_\ast$ independent of $e_8$, was first proved  by Savin \cite[Propposition 1]{s05} for infinity harmonic functions, and later
by Wang-Yu \cite[Propsotion 4.1]{wy}
when $H\in C^2(\rr^2)$ is locally uniformly convex.
Here, we  prove Theorem \ref{xprop} by using Theorem \ref{lla}
and blending some ideas of Savin \cite{s05} with the procedure
in the proof of Wang-Yu \cite[Propsotion 4.1]{wy}.
However, to deal with several essential difficulties
arising from general $H$ as in Theorem \ref{xprop},
we need  a few new ideas:

\begin{enumerate}
\item[(c)] We establish two analytic characterizations of  ({\bf A}) of Theorem \ref{lla}, {\color{red}
which measure a  very weak modulus continuity of $\partial H$ via angles or inner product; } 
see Propositions \ref{LEM4.5} and \ref{e-e6}.
Both 
We also develop several fundamental properties of cone functions of $H$,
see Lemmas \ref{la-mu}, \ref{LEMest} and \ref{xxlem},
and Corollaries \ref{com-linear} and \ref{xew4} for details.

\item[(d)] With the help of the properties in (a), we are able to modify the arguments in \cite{wy} to achieve:
\begin{enumerate}
\item [(d1)] First, as in \cite[Lemma 4.2]{wy} one gets  an auxiliary vector $e$ from the linear approximation property \eqref{LAP}
and Savin's planar topology argument, see Lemma \ref{xcon-com}.
The analytic properties of $H$ and cone functions given by (a)
allow us to construct a discrete gradient flow, which yields
the length estimate (see Lemma \ref{xlength}): there exists $\eta(\ez)\rightarrow 0$ such that
$$H(e_8)\le H(e)+\eta(\ez), \ \mbox{provided $\dz>0$ is small}.$$
This, combined with $H(e )\le H(e_{0,8})+C\dz$, implies that
$|H(e_8)-H(e)|\le \eta(\ez)$ and  $|H(e_{0,8})-H(e )|\le \eta(\ez)$.
\item [(d2)] Next, we will show that
$$|e_8-e|\le\frac{\ez}2\ \mbox{and}\ |e_{0,8}-e|\le\frac{\ez}2$$
so that \eqref{xlin3x} follows.
To this end, we establish
 an angle estimate, in terms of $\eta (\ez)$, between the vector $e_8-e$ (resp. $e_{0,8}-e$) and some vector
 $q\in\partial H(p)$ with $|p-e|<\eta$ and $H(p)=H(e)$. This is the content of  Lemma \ref{xangle} (resp. Lemma \ref{xangler}). By suitably choosing $\eta(\ez)>0$, the above norm estimate follows from Proposition \ref{e-e6}, Lemma \ref{xangle}
(resp. Lemma \ref{xangler}).
The angles estimates in Lemmas \ref{xangle} and \ref{xangler} will
be proved by applying Proposition \ref{LEM4.5} and some planar topology.
\end{enumerate}
 \end{enumerate}

We would like to remark that the angle estimates  in Lemmas \ref{xangle}
and \ref{xangler} play essential roles in the proof of Theorem 6.1.  In fact, without these angle estimates, we can only obtain
$|H(e_8)-H(e_{0,8})|\le  \ez$ in Theorem \ref{xprop}. 
So,  instead of  everywhere differentiability of $u$ and the modulus continuity of $Du$ as in \eqref{rho} of Theorem \ref{c2}, we can only
get the modulus  continuity of $Su$. 
However,
the modulus  continuity of $Su$ is weaker than the everywhere differentiability of $u$ and the modulus continuity of $Du$.

 In a recent prepreint \cite{fmz}, the authors are able to employ
 Theorem \ref{lla0} above  to establish in dimension two, the Sobolev $W^{1,2}_\loc$-regularity of
 absolute minimizer  $u$ of $[H(Du)]^\alpha$ for all $\alpha>0$ when $H\in C^2(\rr^2)$ is locally strongly convex
 or  $\alpha>\tau_H(0)$  when $H\in C^0(\rr^2)$ is locally strongly convex.
 In  another forthcoming paper \cite{fmz2}, the authors further apply Theorem \ref{lla0} to
 study the differentiability of absolute minimizers in dimensions $n\ge3$, when $H\in C^0(\rn)$ is locally strongly convex.

\section{Preliminaries}

In this section, we will collect all the basic properties on absolute minimizers
that are necessary to our main theorems, for which we follow \cite{acjs}
closely.

Recall that any linear function is an absolute minimizer of $H$. The first
property is the comparison principle among absolute minimizers, established by
\cite{acjs}.

\begin{lem}\label{com} For $n\ge 2$, assume
 $H$ satisfies both \eqref{assh0} and ({\bf A}) of Theorem \ref{lla}.
For any domain $  U\subset\rn$,  if $u,v\in AM_H(U)\cap C^0(\overline U) $
then we have
\begin{equation}\label{AM1}
\max_{  x\in U}\big(\pm u(x)-v (x)\big)
\le \max_{x\in \partial U}\big(\pm u(x)-v(x)\big).
\end{equation}
\end{lem}

Next we recall the property of comparison with cones for absolute minimizers.
Assume $H$ satisfies \eqref{assh0}. For any $k\ge 0$, a cone function of $H$, $C^H_k(\cdot)$ is defined by
$$C^H_k(x)=\sup_{H(p)\le k} p\cdot x, \quad \forall x\in\rn.$$
It is evident that $C^H_k(\cdot)\in C^{0,1}(\rn)$ is convex, positively homogeneous
of degree one, sub additive,
and $C^H_k(x) > 0$ for all $k> 0$ and $x\ne 0$.

The proof of following lemma can be found by \cite[Lemma 2.18]{acjs}.
\begin{lem} \label{LEM7.11}  For $n\ge 2$,  assume $H$ satisfies
\eqref{assh0}.  For any domain $U\subset\rn$,
 $u\in C^{0,1}(U)$,  and $k\ge0$, the following statements are equivalent:
\begin{enumerate}
\item[(i)] $H\left(Du\right)\le k$ almost everywhere in $U$.
\item[(ii)] $u\left(x\right)-u\left(y\right) \le C^H_k\left(x-y\right)$,
provided the line segment $[x,y]\subset U$.
\end{enumerate}
\end{lem}
Denote by ${\rm{USC}}\left(U\right)$ (resp. ${\rm{LSC}}\left(U\right)$)
the space of upper (resp. lower) semicontinuous functions in $U$.
We introduce
\begin{defn}\rm For $n\ge 2$, assume
$H\in C^0(\rn)$ satisfies \eqref{assh0}.
\begin{enumerate}
\item[(i)] A function $u\in {\rm{USC}}\left(U\right)$ satisfies the property of comparison with cones  for $H$ from above in $U$, if
$$\max_V\{u- C^H_k\left(x-x_0\right)\}=\max_{\partial V}\{u- C^H_k\left(x-x_0\right)\}$$
whenever $V\Subset\Omega$, $k\ge0$ and $x_0\in\rn\setminus V$.
Write $u\in CCA_H\left(U\right)$ for brevity.
\item[(ii)] A function $u\in {\rm{LSC}}\left(U\right)$ satisfies the property of comparison  with cones  for $H$ from below in $U$, if
  $$\min_V\{u+ C^{ H}_k\left(x_0-x\right)\}=\min_{\partial V}\{ u+ C^{  H}_k\left(x_0-x\right)\}$$
whenever $V\Subset\Omega$, $k\ge0$ and $x_0\in\rn\setminus V$.
Write $u\in CCB_H\left(U\right) $ for brevity.
\item[(iii)] A $u \in C^0\left(U\right)$ satisfies the property of comparison  with cones  for $H$   in $U$ (for brevity, write $u\in CC_H\left(U\right)$),
 if  $u\in CCB_A\left(U\right)\cap CCB_H\left(U\right) $.
 \end{enumerate}
\end{defn}

It is straightforward to see that a function in $CC_H(U)$ enjoys the following stability property.
\begin{lem}\label{LEMconverges} For $n\ge 2$, let $H$ satisfy \eqref{assh0}.
Assume $u_j\in CC_H(U) $ and $ u_{j}\to u_\fz$ in $C^0(U)$.
 Then $ u_\fz\in CC_H(U)$.
 \end{lem}

 \begin{proof} For simplicity, we only show that $u_\fz\in CCA_H(U)$.
 To do it, let $V\Subset U$ and $x^0\notin V$ and assume
 that for some $k\ge 0$,
 $$u_\fz(x)-u_\fz(x_0)\le C_k^H(x-x_0)+b, \ \forall x\in \partial V,$$
 then we have that for any $\ez>0$, if $j$ is sufficiently large then
 $$u_j(x)-u_j(x_0)\le C_k^H(x-x_0)+b+\ez, \ \forall x\in \partial V.$$
 Since $u_j\in CC_H(U)$,  it follows that
 $$u_j(x)-u_j(x_0)\le C_k^H(x-x_0)+b+\ez,
 \forall x\in   V.$$
Sending $j\rightarrow\infty$,
we obtain that
$$u_\fz(x)-u_ \fz(x_0)\le C_k^H(x-x_0)+b +\ez,\ \forall x\in   V.$$
Since $\ez>0$ is arbitrary, it follows that $u_\fz\in CCA_H(U)$.
\end{proof}

Let $L$ be the convex conjugate of $H$ given by \eqref{conj}. If $H$ satisfies \eqref{assh0}, then $L$ satisfies \eqref{assh0},
and $H$ is also the convex conjugate of $L$.
Given a domain $U\subset\rn$ and a bounded function $u\in C^0 \left(U\right)$, the Hamilton-Jacobi flows are  defined by
$$T^tu\left(x\right)=\sup_{y\in U}\left\{u\left(y\right)-tL\left(\frac{y-x}t\right)\right\},
\ \mbox{and}\  T_tu\left(x\right)=\inf_{y\in U}\left\{u\left(y\right)+tL\left(\frac{x-y}t\right)\right\}, \quad\forall t>0,x\in U,$$
and
$$T^0u\left(x\right)=u\left(x\right)=T_0u\left(x\right),\forall x\in U.$$
The slope functions via the Hamilton-Jacobi flows can be defined by
$$S^+_t u\left(x\right)=  \frac1t\big\{T^tu\left(x\right)-u\left(x\right)\big\} \ \mbox{and}\  S^-_t u\left(x\right)= \frac1t\big\{T_tu\left(x\right)-u\left(x\right)\big\}\quad\forall x\in U, t>0.$$

 For any $r>0$, set $U_r:=\{x\in U: \dist\left(x,\partial U\right)>r\}$.
\begin{defn} \rm  For $n\ge 2$, assume $H$ satisfies \eqref{assh0}.
For any domain $U\subset\rn$,
\begin{enumerate}
\item[(i)] a bounded function $u\in C^0\left(U\right)$ enjoys the convexity
criteria,  if
for any $r>0$ there exists a $\dz_r>0$ such that   for all
$x\in U_r$, the map $T^t u(x): [0,\dz_r)\mapsto \mathbb R$ is convex in
the $t$-variable.

\item[(ii)] a bounded function $u\in C^0\left(U\right)$ enjoys the concavity criteria,
if for any $r>0$ there exists a $\dz_r>0$ such that   for all
$x\in U_r$, the map $T_tu(x): [0,\dz_r)\mapsto \mathbb R$  is concave
in the $t$-variable.
\end{enumerate}
\end{defn}

We point out  that when $U=\rn$, $T^tu , T_tu$, and $S^\pm_t u$
can be defined for any $u\in C^{0,1}(\rn)$ that satisfies   $\|H(Du)\|_{L^\fz(\rn)}< \fz$, due to the fact that $L$ satisfies \eqref{assh0}.
Moreover, we have the following localization property for both $T^t u$ and
$T_tu$.
\begin{lem}\label{LEM3.3} For $n\ge 2$, assume $H$ satisfies \eqref{assh0}.
If $u\in C^{0,1}(\rn)$ satisfies $\| H(Du) \|_{L^\fz(\rn)}=k<\fz$, then there exists a constant $R_k>0$ depending on $k$ and $H$ such that
 \begin{equation}\label{E3.1}\pm S^\pm_tu\left(x\right)=\sup_{y\in \overline {B\left(x,R_kt\right)}} \left\{\pm \frac{u\left(y\right)-u\left(x\right)}t- L\big(\pm\frac{y-x}t\big)\right\}\quad\forall x\in\rn, t>0.
 \end{equation}
\end{lem}

\begin{proof} From Lemma 2.2, there exists $R_k>0$ such that
$$|u\left(y\right)-u\left(x\right)|\le R_k|x-y|, \ \forall x,y\in\rn.$$
This, combined with the superlinear growth of $L$, implies
that there exists a monotone increasing function $M:[0,\fz)\to[0,\fz)$ such that $\lim_{R\rightarrow\infty}M\left(R\right)=\fz$ and
$L\left(q\right)\ge M\left(R\right)R$ whenever $|q|\ge R$.
If $M\left(R_k\right) > k$ and $|x-y|\ge R_kt$, then
$$\pm \frac{u(y)-u(x)}t-L\big(\pm\frac{y-x}t\big)\le
\big(k-M\big(\frac{|x-y|}t\big)\big)\frac{|x-y|}t\le 0,$$
which yields \eqref{E3.1}.  This completes proof of Lemma \ref{LEM3.3}.
\end{proof}

Now we state the most important characterization of absolute minimizers
in terms of comparison with cones and convexity/concavity criteria.
Since the condition ({\bf A}) of Theorem \ref{lla}
implies that the minimal level set of $H$
has an empty interior, the proof follows directly from \cite[Theorem 4.8]{acjs},
which is omitted here.

\begin{lem}\label{LEM7.13} For $n\ge 2$,
assume $H$ satisfies both \eqref{assh0}  and ({\bf A}) of Theorem \ref{lla}.
Then, for any domain $U\subset\rn$ and bounded function $u\in C^0(U)$, the following statements are equivalent:

\begin{enumerate}
\item[(i)]  $u\in  AM_H\left(U\right)$.
\item[(ii)]  $u\in  CC _H\left(U\right)$.
\item[(iii)] $u\in C^0\left(U\right)$  enjoys both the convexity criteria
and concavity criteria.
\end{enumerate}
  \end{lem}

It follows from Lemma \ref{LEM7.13} and \cite[Lemma 4.2]{acjs} that
if $u\in AM_H\left(U\right)$ is bounded, then for $r>0$ and $x\in U_r$,
the function $t\in(0,\dz_r]\to \pm S^\pm_t u\left(x\right)$ is monotone
increasing.
Hence
\begin{equation}\label{slope1}
S^\pm u(x)=\lim_{t\to0}S^\pm_t u\left(x\right), \quad\forall x\in U,
\end{equation}
exists and is upper semicontinuous in $U$.
Moreover, as in \cite[Lemma  4.3]{acjs},
for any $V\Subset U$, it holds that
$$\|S^+ u\|_{L^\fz(V)}=\|S^- u\|_{L^\fz(V)}=\|H(Du)\|_{L^\fz(V)},$$
and hence
\begin{equation}\label{E-2}
S u\left(x\right) :=\lim_{r\to0}\|H\left(Du\right)\|_{L^\fz\left(B\left(x,r\right)\right)}=\lim_{r\to0}\|S^\pm u\|_{L^\fz(B(x,r))}=\pm S^\pm u(x)
\end{equation}
holds for all $x\in U$.

We also recall the slope functions  defined via the cone functions:
$$\widehat{S}^+_t u\left(x\right)=\inf\left\{k\ge0, u(y)-u(x)\le C^H_k(y-x)\quad\forall y\in \partial B(0,t)\right\}, \quad \forall x\in U,\ 0<t<\dist(x,\partial U) $$
and
$$-\widehat {S}^-_t u\left(x\right)=\inf\left\{k\ge0, u(x)-u(y)\le C^H_k(y-x)\quad\forall y\in \partial B(0,t)\right\}, \quad \forall x\in U,\ 0<t<\dist(x,\partial U).$$

Following the argument of \cite[Proposiitons 3.1 and 3.3]{gwy06} line by line, we have the following Lemmas \ref{LEMcone} and \ref{LEMconeincreasing}, which will be needed in the proof of Theorem \ref{lla0}.
\begin{lem}\label{LEMcone} For $n\ge 2$, let $H$ satisfy \eqref{assh0}.
For $U\Subset\rn$, assume that $u\in CC_H(U)$. Then for any $x\in U$,
the functions $t\in(0,\dist(x,\partial U))\mapsto \pm \widehat S^\pm_tu(x)$ is monotone increasing,
and
 $ \displaystyle \widehat{S}u(x)= \lim_{t\to0}\pm\widehat S^\pm_tu(x)$
 exists for all $x\in U$ and is upper semicontinuous in $U$.
 Furthermore, $\widehat{S}u(x)=Su(x)$ for all $x\in U$.
\end{lem}

\begin{proof} For $0<t<\dist(x,\partial U)$, it follows from
 $u\in CC_H(U)$ that
 $$\mbox{$u(y)\le u(x)+{C}^H_{\widehat {S}^+_t u(x)}(y-x),
 $
 for  $ y\in \overline {B(x,t)}$. }$$
Hence we have that, for $0<r<t$,
$$u(y)\le u(x)+{C}^H_{\widehat S^+_t u(x)}(y-x),\quad
\forall y\in {\overline B(x,r)}$$
so that
$\widehat S^+_{r} u(x)\le \widehat S^+_{t} u(x)$, and
the function $t\in(0,\dist(x,\partial U))\mapsto  \widehat {S}^+_tu(x)$
is monotone increasing. Therefore,
$\widehat S^+ u(x):=\lim_{t\to 0} \widehat S^+_{t} u(x)$
exists and  is upper semicontinuous in $  U$.

To see $\widehat S^+ u=Su$ in $U$,
let $k_t= \|H(Du)\|_{L^\fz(B(x,t))}$. Then by Lemma \ref{LEM7.11} we have
$$u(y)\le u(x)+{C}^H_{k_t }(y-x),
\quad \forall y\in \overline {B(x,t)}.$$
It is not hard to see this implies
that $S^+_{t} u(x)\le k_t$, and hence
$$\widehat{S}^+u(x)=\lim_{t\to0} \widehat{S}^+_{t} u(x)
\le \lim_{t\rightarrow 0} k_t=Su(x).$$
On the other  hand,
the upper semicontinuity of $\widehat{S}^+u$  implies
that for any $\ez>0$, there exists $t_\ez>0$ such that
$$\widehat S^+_{t_\ez}u(y)\le \widehat S^+u(x)  + \ez,
\forall y\in \overline {B(x,t_\ez)}.$$
Therefore, for any $z,y\in B(x,\frac{t_\ez}2)$,
we have
$$u(z)\le u(y)+{C}^H_{\widehat S^+_{t_{\ez}}u(y) }(z-y)
\le u(y)+{C}^H_{\widehat S^+u(x)+\ez }(z-y).$$
Applying Lemma \ref{LEM7.11} again, we conclude that
 $$\|H(Du)\|_{L^\fz(B(x,\frac{t_\ez}2))}\le \widehat S^+u(x)+\ez.$$
This, after sending $\ez\to0$, implies that
$Su(x)\le \widehat{S}^+u(x)$.  Thus $Su(x)= \widehat{S}^+u(x)$
for $x\in U$.
Similarly, we can also show $Su(x)= -\widehat{S}^-u(x)$ for $x\in U$.
\end{proof}

 \begin{lem}\label{LEMconeincreasing} For $n\ge 2$,
 assume $H$ satisfies \eqref{assh0}. For $U\Subset\rn$,
let $u\in CC_H(U)$. Then for any $x\in U$ and $0<r<\dist(x,\partial U)$,
 there exists $x_r\in\partial B(x,r)$ such that
 $$u(x_r)-u(x)=C^H_{\widehat {S}^+_{r}u(x)}(x_r-x)
 \ \mbox{and}\ \widehat{S}u(x_r)\ge \widehat S^+_{r}u(x).$$
 \end{lem}
\begin{proof}  By definition of $C^H_k(\cdot)$, there exists
$x_r\in\partial B(x,r)$ such that
$$u(x_r)-u(x)={C}^H_{\widehat S^+_{r}u(x)}(x_r-x).$$
For $0<\theta\le 1$, let $x_\theta=\theta x+(1-\theta)x_r$.
Then we have
\begin{align}\label{sss1}
u(x_\theta)\le u(x)+{C}^H_{\widehat S^+_r u(x)} (x_\theta-x)
=u(x)+(1-\theta){C}^H_{\widehat S^+_r u(x)} (x_r-x)=
u(x_r)-\theta {C}^H_{\widehat S^+_r u(x)} (x_r-x).
\end{align}

On the other hand, for any $0<R<{\rm dist}(x_r,\partial U)$
we can choose a sufficiently small $0<\theta< \frac{R}{2r}$ such that
\begin{align}\label{sss2}
u(x_r)\le u(x_\theta)+{C}^H_{\widehat S^+_{R-\theta r} u(x_\theta)}(x_r-x_\theta)=
u(x_\theta)+\theta{C}^H_{\widehat S^+_{R-\theta r} u(x_\theta)}(x_r-x).
\end{align}
Combining \eqref{sss2} with \eqref{sss1}, we obtain
$$\widehat S^+_r u(x)\le \widehat S^+_{R-\theta r} u(x_\theta)\le
\widehat S^+_{R} u(x_\theta).$$
Sending $\theta\to 0$ first and then $R\to 0$, we conclude that
 $$Su(x_r)=\widehat S^+ u(x_r)\ge \widehat S^+_r u(x).$$
 This completes the proof.
\end{proof}

\begin{cor}\label{llg} For $n\ge 2$, assume $H$ satisfies \eqref{assh0}.
If $u\in CC_H(\rn)$ has a linear growth at the infinity, then
$\|H(Du)\|_{L^\fz(\rn)}<\fz$.
\end{cor}
\begin{proof}
Since there exists $K>0$ such that $|u(x)|\le K(1+|x|)$ for all $x \in\rn$,
we have
$$\mbox{$|u(x)-u(y)|\le  K(2+|x|+|y|)\le (2K+1)|x-y|$
whenever $x\in\rn$ and $ |y-x|=|x|+1$}.$$
 Since there exists $k>0$, depending on
 $K$ and $H$, such that
 $$(2K+1)|z|\le C^H_{k}(z), \forall z\in\rn,$$
we obtain that
$$|u(x)-u(y)|\le  C^H_{k}(y),\ \forall y\in \partial B(x, |x|+2),$$
it follows from $u\in CC_H(\rn)$ that
$$|u(x)-u(y)|\le  C^H_{k}(y),\ \forall y\in B(x, |x|+2).$$
This, combined with Lemma  \ref{LEM7.11}, implies
 $$H(Du)(x)\le k, \mbox{ for almost every } y\in B(x, |x|+2).$$
 Thus  $\|H(Du)\|_{L^\fz(\rn)}\le k<\fz$.
\end{proof}

\section{Geometric and analytic properties of $H$,  $L$, and cone functions}

In this section, we will develop
both geometric and analytic properties of $H$ and $L$, and
$C^H_k(\cdot)$.

\subsection{Properties of $H$ and $L$}
For $n\ge 2$, assume $H$ satisfies \eqref{assh0}.
Let $L$ be the convex conjugate of $H$. Then $L$ also satisfies \eqref{assh0}.
For any $q\in\rn$, denote by $\partial L(q)$ the subdifferential set of $L$ at $q$,
 that is,
$$p\in\partial L(q) \Longleftrightarrow L(q')-L(q)\ge \langle p, q'-q\rangle,
\ \forall q'\in\rn.$$
The subdifferential set of $H$, $\partial H(p)$,  at $p\in\rn$, can be similarly defined.

Recall that in $\rn$, a $1$-simplex is a line segment, or the convex hull of $2$ distinct points, and for $2\le d\le n$, a $d$-simplex is  the convex hull of a $(d-1)$-simplex  and a point, which is not contained in the $(d-1)$-dimensional affine plane determined by the $(d-1)$-simplex.

Proposition \ref{LEMC3} on the geometric characterization of $\partial L$,
when $H$ is not constant on any $d$-simplex for $d=1,2$,
plays a key role in the proof of Theorem \ref{LEMC4}.

 \begin{prop}\label{LEMC3} For $n\ge 2$, assume
 $H$ satisfies \eqref{assh0}. Then
 \begin{enumerate}
 \item[(i)]   The following statements are equivalent:
    \begin{enumerate}
   \item[(i-a)] $ H$ is not constant in any $1$-simplex.
    \item[(i-b)] for any $q\in\rn$,  $\partial L\left(q\right)$
is either a single point, or a line segment on which
$H$ is strictly monotone.
 \end{enumerate}
  \item[(ii)]   The following statements are equivalent:
    \begin{enumerate}
   \item[(ii-a)]
   $ H$ is not constant in any $2$-simplex.
\item[(ii-b)] for any $q\in\rn$,  $\partial L\left(q\right)$ must be  one of the following:
   \begin{enumerate}
  \item[(ii-b-1)] a single point;
  \item[(ii-b-2)] a bounded closed line segment;
  \item[(ii-b-3)] a bounded closed convex set in a $2$-dimensional affine plane,
  whose boundary consists of $4$ simple ``curves"
  $\gz_0,\gz_1,\gz_2,\gz_3$ oriented in order so that
  $\gz_0$ (resp. $\gz_2$) is either a single point
  or a line-segment on which $H$ attains the minimum (resp. maximum)
  in  $\partial L(q)$, and $\gz_1$ (resp.  $ \gz_3$) is such
  that $H$ is strictly monotone increasing (resp. decreasing).
    \end{enumerate}
    \end{enumerate}
\end{enumerate}
 \end{prop}
When $H$ is strictly convex,  we have
\begin{prop}\label{LEMstrcon} For $n\ge 2$,
assume $H$ satisfies \eqref{assh0}. The following statements are equivalent:
    \begin{enumerate}
   \item[(i)] $ H$ is strictly convex.
   \item[(ii)] For any $q\in\rn$, $\partial L(q)$ is a singleton.
    \item[(iii)] $L\in C^1(\rn)$.
 \end{enumerate}
 \end{prop}
We will establish in Propositions \ref{e-e6} and \ref{LEM4.5}
some analytic characterization of $H$, when $H$ is not
constant in any line segment, which will play an important role
in the proof of Theorem \ref{xprop}.

\begin{prop}\label{e-e6} For $n\ge 2$, assume
$H$ satisfies \eqref{assh0}. Then the following statements are equivalent:

\begin{enumerate}
\item[(i)]    $H$ is not constant in any line segment.

\item[(ii)] For each  $R\ge 1$ and each $\ez>0$, there exists
$\psi_R \left( \ez\right)>0 $  such that for any
$  v\in \overline{B\left(0,R\right)}$,
if
\begin{equation}\label{e-e6angle}
H\left(p+v\right)-H\left(p\right)\le \psi_R \left( \ez\right)\ \mbox{and}\ |\measuredangle \left(q,v\right)-\frac{\pi}2|  \le \psi_R \left( \ez\right)
\end{equation}
for some  $p\in \overline{B\left(0,R\right)},\  q\in\partial H\left(p'\right)$ and  $p'\in \overline{B(p,\psi_R(\ez))}$,  then  $|v|\le\ez.$ Here $\measuredangle(q,v)$
denotes the angle between $q$ and $v$.
  \end{enumerate}
 \end{prop}

 For $R>0$, assume that $\psi_{R}(\epsilon)\ge0$ is
 monotone increasing and satisfies
 $\psi_R(\ez)\le \ez^2, \forall \ez\in(0,1)$.

\begin{prop}\label{LEM4.5} For $n\ge 2$, assume
 $H$ satisfies \eqref{assh0}. Then the following statements are equivalent:

\begin{enumerate}
\item[(i)]    $H$ is not  constant in any line segment.

\item[(ii)]

For each $R\ge1$ and each $\eta>0$,  we have
$$ \phi_R(\eta)= \min\left\{(p-e)\cdot \frac{q}{|q|}: H(p)=H(e)\le R , |p-e|\ge \eta, q\in\partial H(p)\right\}>0$$
\end{enumerate}
\end{prop}

In order to prove the above results, we recall some basic properties
of $H$.

 \begin{lem}\label{LEM3.1} For $n\ge 2$, assume $H$ satisfies \eqref{assh0}.
 Then we have
 \begin{enumerate}
 \item[(i)] For any $p,q\in\rn$,
 $$q\in \partial H\left(p\right) \Longleftrightarrow
  H\left(p\right)+L\left(q\right)=\langle p,q\rangle  \Longleftrightarrow p\in \partial L\left(q\right).$$
In particular, $0\in\partial H(p)$ if and only if $H(p)=0$,
and $0\in\partial L(q)$ if and only if $L(q)=0$.
  \item[(ii)] If $p_1,p_2\in \partial L\left(q\right)$ for some $q\in\rn$, then
  $$\mbox{$\lz p_1+\left(1-\lz\right)p_2\in \partial L\left(q\right)$  and $H(\lz p_1+\left(1-\lz\right)p_2)=\lz H(p_1)+(1-\lz)H(p_2)$ for all $\lz\in[0,1]$.}$$
  \item[(iii)] For $q_1,q_2\in\rn$, if there exists $\lambda\in (0,1)$ such that
  $$\mbox{$L\left(\lz q_1+\left(1-\lz\right)q_2\right)=\lz L\left(q_1\right)+\left(1-\lz\right) L\left(q_2\right)$}$$
  and $p \in \partial L\left(\lz q_1+\left(1-\lz\right)q_2\right)$, then
  $p\in \partial L\left(q_1\right)\cap \partial L\left(q_2\right)$.
   \item[(iv)] The set $\partial L\left(q\right)$ is bounded locally uniformly in $q\in\rn$.  If
 $p_i\in \partial L\left(q_i\right)$ for all $i\in\nn$ and $q_i\to q_0$ as $i\to\fz$,
 then there exists $p_0\in \partial L\left(q_0\right)$ such that,
 after passing to a subsequence,  $p_i\to p_0$ as $i\to\fz$.
In particular, $\partial L(q)$ is a closed subset of $\rn$ for any $q\in\rn$.
    \end{enumerate}
\end{lem}

\begin{proof} (i) Note that $q\in \partial H\left(p\right)$  if and only if
$$\langle p,q\rangle-H\left(p\right)\ge \langle q,p' \rangle-H\left(p'\right), \forall p'\in\rn.$$
While
$H\left(p\right)+L\left(q\right)=\langle p,q\rangle$ if and only if
$$ \langle p,q\rangle-H\left(p\right)\ge \langle q, \hat{p} \rangle-H\left(\hat{p}\right), \forall \hat{p}\in\rn.$$
Thus $q\in \partial H\left(p\right)$ if and only if $H\left(p\right)+L\left(q\right)=\langle p,q\rangle$.
Similarly, we can show that $p\in \partial L\left(q\right)$ if and only if $H\left(p\right)+L\left(q\right)=\langle p,q\rangle$.

\smallskip
\noindent (ii) If $p_1,p_2\in \partial L\left(q\right)$ for some $q\in\rn$,
then by using   (i) we have that for any $\lz\in [0,1]$,
\begin{align*}
L\left(q\right)+ H \left(\lz p_1+ \left(1-\lz\right)p_2\right)&\ge \left(\lz p_1+ \left(1-\lz\right)p_2\right)\cdot q\\
&=\lz p_1\cdot q+\left(1-\lz\right)p_2\cdot q\\
&=\lz (H\left(p_1\right)+L(q))+\left(1-\lz\right) (H\left(p_2\right)+L\left(q\right)),\\
&=\lz H\left(p_1\right)+\left(1-\lz\right) H\left(p_2\right)+L\left(q\right)
\end{align*}
so that
$$H \left(\lz p_1+ \left(1-\lz\right)p_2\right)\ge \lz H\left(p_1\right)+\left(1-\lz\right) H\left(p_2\right), \forall \lz\in[0,1].$$
This, combined with the convexity of  $H$, implies
 $$H \left(\lz p_1+ \left(1-\lz\right)p_2\right)= \lz H\left(p_1\right)+\left(1-\lz\right) H\left(p_2\right) \forall \lz\in[0,1].$$
Hence $\left(\lz p_1+ \left(1-\lz\right)p_2\right)\in\partial L\left(q\right)$.

\medskip
\noindent (iii) If $L\left(\lz q_1+\left(1-\lz\right)q_2\right)=\lz L\left(q_1\right)+\left(1-\lz\right) L\left(q_2\right)$ and
 $p \in \partial L\left(\lz q_1+\left(1-\lz\right)q_2\right)$ for some $\lz\in\left(0,1\right)$, then by (i)  we get
 \begin{align*}
  L\left(\lz q_1+\left(1-\lz\right)q_2\right)
&= \left(\lz q_1+ \left(1-\lz\right)q_2\right)\cdot p-  H \left(p\right)\\
&=\lz (q_1\cdot p-  H \left(p\right))+\left(1-\lz\right)(q_2\cdot p-  H \left(p\right))\\
&\le \lz L(q_1)+(1-\lz)L(q_2).
\end{align*}
Hence we have that $q_i\cdot p=H\left(p\right)+L\left(q_i\right) $ for $i=1,2$,
which implies that $p\in\partial L(q_1)\cap\partial L(q_2)$.

\medskip
\noindent (iv) For $R\ge 1$, if $|q|\le R$ and $p\in \partial L\left(q\right)$,
then we have
$$H\left(p\right)=p\cdot q- L\left(q\right)\le |q||p|-L\left(q\right)
\le C_1\left(R\right)+ R|p|.$$
This, combined with the superlinear growth of $H$,
implies that there exists $C_2(R)>0$
such that $|p|\le C_2\left(R\right)$. From this, we see that
if $p_i\in \partial L\left(q_i\right)$ for $i\in\nn$ and $q_i\to q_0$ as $i\to\fz$,
then $p_i$ is bounded, and
$$p_i \cdot q_i=  H\left(p_i\right)+L\left(q_i\right), \ i\in\nn.$$
Hence, up to a subsequence, there exists $p_0\in\rn$
such that $p_i$ converges to $p_0$ as $i\to\fz$.
By the continuity of $H$ and $L$,  we then have
$$p_0 \cdot q_0=  H\left(p_0\right)+L\left(q_0\right).$$
By (i), this implies   $p_0\in\partial L\left(q_0\right)$.
The proof of Lemma \ref{LEM3.1} is now complete.
\end{proof}

As a consequence of Lemma \ref{LEM3.1}, we have
 \begin{cor}\label{propH} For $n\ge 2$, assume
 $H$ satisfies \eqref{assh0} and ({\bf A}) of Theorem \ref{lla}.
 Then $0\in \partial H(0) $,
 $H(p)>0$ and  $0\notin \partial H(p)$ for all $p\in \rr^n\backslash\{0\}$.
  \end{cor}
\begin{proof} If $H(p_0)=0$ for some $p_0\not=0$, then
 by convexity of $H$ and $H(0)=0$, we have that $H(p)=0$ for
 all $p\in [0,p_0]$, which contradicts to ({\bf A}) of Theorem \ref{lla}.
 Thus $H(p)>0$ and $0\notin \partial H(p)$ whenever $p\ne0$.
Moreover, by Lemma \ref{LEM3.1},
 $0\in \partial H(p)$ if and only if $H(p)=0$ if and only if  $p=0$.
\end{proof}

Lemma \ref{LEMconst} provides some geometric and analytic properties
of $H$,  if $H$ is constant in some line segment.

\begin{lem}\label{LEMconst} For $n\ge 2$, assume
$H$ satisfies \eqref{assh0}. Given a pair of points $a, b\in\rn$
with $a\not=b$, the following statements are equivalent:
   \begin{enumerate}
   \item[(i)]  $H$ is constant in the line segment $[a,b]$.

   \item[(ii)] $b-a\perp  \partial H\big(\frac{a+b}{2}\big)$ and
   \begin{equation}\label{Hcon} \partial H\big(\frac{a+b}{2}\big)
   =\partial H(\lz a+(1-\lz)b)\subset\partial H(a)\cap \partial H(b),
   \quad \forall \lz\in(0,1).
    \end{equation}

    \item[(iii)] There exists a $q\in\rn$ such that $b-a\perp q$ and $[a,b]\subset\partial L(q)$.

    \item[(iv)] There exists a $q\in\rn$ such that $H(a)=H(b)$ and $[a,b]\subset\partial L(q)$.

    \end{enumerate}
     \end{lem}

    \begin{proof}
(i)$\Rightarrow$(ii): Assume that $H$ is  constant in the line segment $[a,b]$. For any $\lz\in(0,1)$,  if $q_\lz\in \partial H(\lz a+(1-\lz)b)$, then
$$\begin{cases} 0=H(a)-H(\lz a+(1-\lz)b)\ge (1-\lz)q_\lz\cdot (a-b)\\
0=H(b)-H(\lz a+(1-\lz)b)\ge \lz q_\lz\cdot (b-a).
\end{cases}
$$
This implies $q_\lz\perp (b-a)$.  Thus for any $\mu\in[0,1]$
and any $p\in\rn$, we have
 \begin{align*} H(p)-H(\mu a+(1-\mu)b)&= H(p)-H(\lz a+(1-\lz)b)\\
 &\ge
  q_\lz\cdot (p- (\lz a+(1-\lz)b))\\
 &= q_\lz\cdot (p- (\mu a+(1-\mu)b))+
  (\mu-\lz)q_\lz\cdot(a-b)\\
  &= q_\lz\cdot (p- (\mu a+(1-\mu)b)),
  \end{align*}
  this implies that $q_\lz\in \partial H(\mu a+(1-\mu)b)$.
  In particular, \eqref{Hcon} holds.

\medskip
\noindent (ii)$\Rightarrow$(iii):
 Let $q\in   \partial H(\frac{a+b}2)$. Then by (ii)  and Lemma \ref{LEM3.1} (i),
 we have $[a,b]\in\partial L(q)$ and $a-b\perp q$.

\medskip
\noindent (iii)$\Rightarrow$(iv):
Let $q$ be given by (iii). Then $[a,b]\subset\partial L (q)$ and $q\cdot a=q\cdot b$.
By Lemma \ref{LEM3.1} (i), we have that
$$ H(b)+L(q)=q\cdot b=q\cdot a=H(a)+L(q),$$
which yields $H(a)=H(b)$.

\medskip
\noindent (iv)$\Rightarrow$(i):  Let $q$ be as in (iv).  Then $H(a)=H(b)$
 and $[a,b]\subset\partial L(q)$. By Lemma \ref{LEM3.1} (i),
 this implies that for any $ \lz\in[0,1]$,
        \begin{align*} H(\lz a+(1-\lz)b)+L(q) & =
  q \cdot (\lz a+(1-\lz)b) \\
  &= \lz q\cdot a+(1-\lz)q\cdot b\\
  &=\lz H(a)+(1-\lz)H(b)+L(q)\\
  &=H(a)+L(q),
  \end{align*}
  which implies that $H$ is constant in $[a,b]$.
    \end{proof}

\medskip
 Now we are ready to prove Propositions \ref{LEMC3}, 3.2, 3.3, and \ref{LEM4.5}.

\begin{proof}[Proof of Proposition \ref{LEMC3}] {\it Proof of (i)}:\\
 (i-b)$\Rightarrow$(i-a): Suppose (i-a) were false. Then
 $H$ is constant in some line segment $[a,b]$. By Lemma \ref{LEMconst} (iv),
 there exists $q\in\rn$ such that
 $[a,b]\subset\partial L(q)$. On the other hand,
 since $H $ is constant in $[a,b]$, it follows from (i-b)
 that $\partial L(q)$ is a singleton. We get the desired contradiction.

\medskip
\noindent (i-a)$\Rightarrow$(i-b): assume that $ H$ is not constant in any line segment. For any $q\in \rn $, assume that $\partial L(q)$ contains at
least two points $p_1, p_2 $, with $p_1\ne p_2$.
It suffices to show  $\partial L\left(q\right)$ is contained in the line
determined by $p_1$ and $p_2$.
Indeed, from Lemma \ref{LEM3.1} we know that $\partial L\left(q\right)$ is a bounded convex set. It is clear that any bounded convex  set contained in a line must be a line segment.

Let $p_0\in\partial L\left(q\right)$ be such that $p_0\ne p_1,p_2$.
Then $H(p_i)\ne H (p_j )$ for $0\le i<j\le 2$. For, otherwise,
$H(p_i)= H (p_j )$ for some $0\le i<j\le 2$.  By Lemma \ref{LEMconst},
we then have $H$ is a constant in $[p_i,p_j]$, which is impossible.
Without loss of generality, we may assume  that
$$H(p_0)<H(p_1)<H(p_2).$$
Then there exists $\lz\in (0,1)$ such that
$$H\big(\lz p_0+(1-\lz)p_2\big)=H\big(p_1\big).$$
Since  $ H$ is not  constant in any line segment,
by applying  Lemma \ref{LEMconst} again we  must have
that
$$ p_1= \lz p_0+ (1-\lz )p_2,$$
which implies that $p_1,p_2$ and $p_3$ lies in the same line, that is, $p_0$ must lie in the line determined by $p_1$ and $p_2$.

\medskip
\noindent{\it Proof of (ii)}. \\
(ii-b)$\Rightarrow$(ii-a): Suppose (ii-a) were false.
Then $H$ is constant in a $2$-simplex $\Delta$, which is the convex hull
of three non-planar points $p_1,p_2,p_3$.
Let $q\in\partial H\big(\frac{p_1+p_2+p_3}3\big)$.
Let $I\subset \Delta$ be any line segment passing through $\frac{p_1+p_2+p_3}{3}$.  Since $H$ is  constant in $I$, it follows from Lemma \ref{LEMconst}
that $I\subset\partial L(q)$. Hence we see that
$\Delta\subset \partial L(q)$ so that  $\partial L(q)$ satisfies neither
(ii-b-1) or nor (ii-b-2).

Now we want to show that  $\partial L(q)$ does not satisfy
(ii-b-3).   For, otherwise,  $\partial L(q)$ is a bounded convex set
whose boundary consists four curves $\gz_i, 0\le i \le 3$, as in (ii-b-3).
Since $H$ is strictly increasing in $\gz_1$ and strictly decreasing in $\gz_3$,
for any $k\in \big(m:=\min_{p\in\partial L(q)} H(p), M:=\max_{p\in\partial L(q)}  H(p)\big)$,
there exist $a_k\in\gz_1$ and $b_k\in\gz_1$ such that
$H(a_k)=H(b_k)=k$. Write $[a_m,b_m]=\gz_0$ and $[a_M,b_M]=\gz_2$.
Thus by Lemma 3.5, we have that $H$ is constant $k$ in $[a_k,b_k]$
and $H^{-1}(k)\cap \partial L(q)=[a_k,b_k]$.  This contradicts
to the fact that $H|_{\Delta}=$constant.

\medskip
\noindent (ii-a)$\Rightarrow$(ii-b): Assume that $H$ is not constant
in any $2$-simplex. We claim that $\partial L(q)$ must be contained  in
a $2$-dimensional affine plane $P\subset\rn$.
For, otherwise, we can find $4$ distinct points $p_0,p_1,p_2, p_3$,
which are not contained in any $2$-dimensional affine plane such that
$\{p_0,p_1, p_2, p_3\}\subset \partial L(q)$. Without loss of generality,
we can assume
$$H(p_0)\le H(p_1)\le H(p_2)\le H(p_3).$$
 Since $H$ is not constant in any $2$-simplex, we must have either\\
 1) $H(p_0)\le H(p_1)< H(p_2)\le H(p_3)$; or\\
 2) $H(p_0)< H(p_1)\le H(p_2)< H(p_3)$.\\

 In  the case 1), we can find three points $p'_0 \in [p_0,p_2], p'_1\in [p_0,p_3],$ and $p'_2\in [p_1,p_2]$ such that
 $$H(p'_0)= H(p'_1)= H(p'_2)=  H\big(\frac{p_1+p_2}2\big),$$
Note that $p'_0,p'_1,p'_2$ are not contained in the same line,
and hence its convex hull is a $2$-simplex, denoted by $\Delta'$.
Moreover,  for any $\lz_i>0$ with  $\displaystyle\sum_{i=0}^2\lz_i=1$
we have that $\displaystyle \sum_{i=0}^2\lz_ip_i' \in \partial L(q)$ and hence
 $$H(\sum_{i=0}^2\lz_ip_i')=(\sum_{i=0}^2\lz_ip_i')\cdot q-L(q)=
  \sum_{i=0}^2\lz_i(p_i' \cdot q-L(q))=\sum_{i=0}^2\lz_iH(p'_i)=H\left(\frac{p_1+p_2}2\right). $$
This implies that $H$ is a constant in $\Delta'$, which contradicts to (ii-a).\\

In the case 2), we can find three points
 $p'_0 \in [p_0,p_1], p'_1\in [p_0,p_2]$ and $p'_2\in [p_0,p_3]$
 such that
 $$H(p'_0)= H(p'_1)= H(p'_2)=  H\left(\frac{p_0+p_1}2\right).$$
Then, similar to the case 1), we can show
that $H$ is  constant in the convex hull of $p'_0p'_1p'_2$,
which is also impossible.

Assume that $\partial L(q)$ is neither a single point nor a line segment.
Then $\partial L(q)$ is a bounded, closed convex domain $U\subset P$.
Thus $\partial L(q)$ is bounded by a simple closed  curve $\gz$.
Note that $H$ achieves its minimum and maximum over $\partial L$
only at the boundary $\partial U=\gz$.
Denote by $\gz_0$ (resp. $\gz_2$) the subset of $\gz$
on which $H$ achieves its minimum (resp. maximum) in $U$.
 Since $H$ is not a constant in any $2$-simplex,
 $\gz_0$ and   $\gz_2$ must be a single point or a line segment. Denote the other two connected components of $\gz\setminus(\gz_0\cup\gz_2)$ by
 $\gz_1$  and $\gz_3$. We may assume that
 the ending point of $\gz_i$ is the starting point of $\gz_{i+1}$ for $i=0,1,2,3$ (where $\gz_4=\gz_0$). Now we want to show that $H$ is strictly increasing along $\gz_1$. For, otherwise, there exists two distinct points
    $p_0,p_1\in\gz_1$ so that $H(p_0)=H(p_1)$. Observe that there exists $p_2\in\gz_3$ with $H(p_2)=H(p_1)$.
Hence $H$ is a constant in the convex hull of  $\{p_0,p_1,p_2\}$, which is a
$2$-simplex, which is impossible.  Similarly, we can show that $H$ is strictly decreasing along $\gz_3$. The proof of Proposition \ref{LEMC3} is now complete.
\end{proof}

\begin{proof}[Proof of Proposition \ref{LEMstrcon}]
(i)$\Rightarrow$(ii): Suppose that  there exists $q\in\rn$
such that $ \partial L(q)$ contains at least two points $p_1, p_2$.
Then  Lemma \ref{LEM3.1}  implies that $H$ is linear in $[p_1,p_2]$,
which contradicts to the strictly convexity of $H$.
Thus for any $q\in\rn$, $\partial L(q)$ must be  a single point.

\medskip
\noindent (ii)$\Rightarrow$(iii):   By Lemma \ref{LEM3.1} (iv),
it suffices to show that $L$ is differentiable everywhere in $\rn$.
We prove by contradiction. Suppose $L$ is not differentiable at $q_0\in\rn$.
By (ii), we have $\partial L(q_0)=\{p_0\}$ is a singleton.
There exist $\ez_0>0$ and a sequence $\{q_i\}\subset\rn\rightarrow q_0$ such that
$$|L(q_i)-L(q_0)-  p_0\cdot (q_i-q_0)|\ge \ez_0|q_i-q_0|.$$
Write  $ \partial L(q_i )=\{p_i \}$.  Then we have
$$ p_0\cdot (q_i-q_0) \le L(q_i)-L(q_0)\le   p_i\cdot (q_i-q_0)=p_0\cdot (q_i-q_0)+ (p_i-p_0)\cdot (q_i-q_0),  $$
so that
$$|L(q_i)-L(q_0)-   p_0\cdot (q_i-q_0)|\le  |p_i-p_0||q_i-q_0|.$$
Thus we obtain that for $i$ sufficiently large,
$$|p_i-p_0|\ge \ez_0.$$
This contradicts to Lemma \ref{LEM3.1} (iv).

\medskip
\noindent (iii)$\Rightarrow$(i):  Assume that $L\in C^1(\rn)$ and
$\partial L(q)=\{DL(q)\}$ for any $q\in\rn$.
Suppose that $H$ is not strictly convex. Then
there would exist $p_1\ne p_2$ in $\rn$ and  $\lz\in\left(0,1\right)$ such that
$$H\big(\lz p_1+(1-\lz)p_2\big)=\lz H(p_1)+(1-\lz) H(p_2).$$
Let $q \in \partial H\left(\lz p_1+\left(1-\lz\right)p_2\right)$. Then by Lemma \ref{LEM3.1} (iii) we have that
$q\in \partial H\left(p_1\right)\cap \partial H\left(p_2\right)$.
Hence by Lemma \ref{LEM3.1}, $p_1,p_2\in\partial L(q)=\{DL(q)\}$,
which is impossible.
\end{proof}

\begin{proof}[Proof of Proposition \ref{e-e6}.]

(ii) $\Longrightarrow$ (i): Suppose (i) were false. Then
$H$ is  constant in a line segment $[a,b]$.
Let  $p=\frac{a+b}{2}$, $q\in\partial H(p)$, and $v=\frac{b-a}2$.
Then $H(p)=H(p+v)$.
By Lemma \ref{LEMconst}, $q\perp v$.
This contradicts to (ii).

\medskip
\noindent (i)$\Rightarrow$ (ii):
Suppose that (ii) were false. There exists $\ez_0>0$
such that for any $i\in\nn$, there exists $v_i\in  \overline{B\big(0,\frac{1}{\ez_0}\big)}$, with $|v_i|\ge \ez_0$,  satisfying
$$H(p_i+v_i)-H(p_i)\le \frac{1}{i},
\ {\rm{and}}\ \big|\measuredangle \left(q_i,v_i\right)  - \frac{\pi}2\big|\le\frac{1}{i},$$
for some $p_i \in \overline{B\big(0,\frac{1}{\ez_0}\big)}$,
$q_i\in\partial H(p_i')$ and $p'_i\in \overline{B(p_i,\frac{1}{i})}$.
It is easy to see that as $i\to\fz$, $p_i\to p_0$
and $p_i'\to p_0$. After passing to a subsequence,
we may assume that there exist $v_0\in\rn$, with $|v_0|\ge\ez_0$,
and $q_0\in\partial H(p_0)$ such that $v_i\to v_0$ and
$q_i\to q_0$ as $i\to\fz$. It is easy to see that
$$H\left(p_0+v_0\right)-H\left(p_0\right)\le 0,\quad \, \measuredangle \left(q_0,v_0\right)  = \frac{\pi}2.$$
This and the convexity of $H$ imply that
$$H\left(p_0+tv_0\right)\le (1-t)H\left(p_0\right)+tH(p_0+v_0)\le H\left(p_0\right)\quad\forall t\in[0,1].$$
On the other hand,  from $q_0\in\partial H(p_0)$ and $ \measuredangle \left(q_0,v_0\right)  = \frac{\pi} 2$, we have
$$ H\left(p_0+tv_0\right)-H\left(p_0\right)\ge q_0\cdot tv_0=0,\quad\forall t\in[0,1].$$
Hence $H\left(p_0+tv_0\right)=H\left(p_0\right)$ for all  $t\in[0,1]$,
This contradicts to (i).  The proof of Proposition \ref{e-e6} is now complete.
\end{proof}

\begin{proof}[Proof of Proposition \ref{LEM4.5}.]
(ii)$\Rightarrow$(i): Assume that  (i) were false.
Then $H$ is constant in a line segment $[a,b]\subset\rn$.
Let  $p=\frac{a+b}2$, $e=b$ and $q\in\partial H(p)$.
Then $H(p)=H(e)$. By Lemma \ref{LEMconst},
$q\perp (p-e)$. Thus  for any $R>1$,
$$\phi_R\big(\frac{|b-a|}2\big)=0.$$
This contradicts to (ii).

\medskip
\noindent (i)$\Rightarrow$(ii): Suppose that (ii) were false.
 Then there exist  $R_0>0$ and $\eta_0>0$  such that $\phi_{R_0}(\eta_0)=0$,
 that is, we can find $p_i$ and $e_i$, with $|p_i-e_i|\ge \eta_0$ and $ H(e_i)=H(p_i)\le R_0$, and $q_i\in\partial H(p_i)$ such that
 $$(p_i-e_i)\cdot \frac{q_i}{|q|}\le \frac{1}{i}.$$
Since $\{p_i\}, \{e_i\}, \{q_i\}$ are  bounded,
we may assume that there exist $p,e,q\in\rn$
such that after passing to a subsequence,
$$p_i\to p,\ e_i\to e, \ {\rm{and}}\  q_i\to q, \ {\rm{as}}\ i\to\fz.$$
It is readily seen that
$$q\in\partial H(p),\ H(p)=H(e)\le R_0, \ |p-e|\ge\eta_0,
\ {\rm{and}}\ (p -e )\cdot \frac{q }{|q|}\le 0.$$
On the other hand, it follows from
 $q\in\partial H(p)$ that
 $$(e-p  )\cdot  q \le    H(e)-H(p)=0.$$
Hence we obtain
$$(e-p  )\cdot  q=0.$$
Applying $q\in\partial H(p)$ and Lemma \ref{LEM3.1}, we have that
$$H(e)+L(q)=H(p)+L(q)=p\cdot q=e\cdot q,$$
so that $e,p\in\partial L(q)$ and $[p,e]\in\partial L(q)$.
Since $H(p)=H(e)$, it follows from Lemma \ref{LEM3.1} (iii)
that $H$ is constant in $[p,e]$, which contradicts to (i).
\end{proof}

\subsection{Properties of cone functions}
In this subsection, we will establish some analytic and geometric properties
of the cone functions $C^H_k(\cdot)$. More precisely, we will prove
the following Lemma  \ref{la-mu}, and  Corollaries \ref{xew4} and \ref{com-linear}.
 \begin{lem} \label{la-mu} For $n\ge 2$,
 assume $H$ satisfies both \eqref{assh0}  and ({\bf A}) of Theorem \ref{lla}.
\begin{enumerate}
 \item[(i)] If  $q\in\partial H\left(p\right)$, then $ C_{H\left(p\right)}^H\left(q\right)=\langle p, q\rangle$.

 \item[(ii)] For $k>0$ and $z\ne 0$, let $p_0\in H^{-1}(k)$ be such that
$C_k^H\left(z\right)=\langle p_0 ,z\rangle $,
then there exists $t_0>0$ such that
$ t_0z\in \partial H\left(p_0\right)$.
\end{enumerate}
 \end{lem}

 \begin{proof}
 (i) If $q\in\partial H\left(p\right)$, then
 $$\langle p, q\rangle =H\left(p\right)+L\left(q\right)=H\left(p\right)+\sup_{p'\in\rn }\big(\langle p', q \rangle- H\left(p'\right) \big)
 \ge  \sup_{H(p')\le H(p)}\langle p', q \rangle
 =C_{H(p)}^H\left(q\right),$$
which,  together with $\langle p, q\rangle\le C_{H(p)}^H\left(q\right)$,
yields $ C_{H(p)}^H\left(q\right)=\langle p, q\rangle$.

\medskip
\noindent (ii)  By Proposition \ref{LEMC3},
  $\partial L(tz)$ is either a single point or a line segment for each $t\ge0$.
   Write $I_t=[a_t,b_t]=H(\partial L(tz))$ for each $t\ge0$.  We claim that
   \begin{equation}\label{union} \displaystyle\cup_{t\ge0}I_t=[0,\fz).
   \end{equation}
 Assume \eqref{union} for the moment. Since $H(p_0)>0$, there
 exist $t_0>0$ and $p_{t_0z}\in\partial L(t_0z)$ such that $H(p_0)=H(p_{t_0z})=k$.  By (i), we have
$$C_{H(p_{t_0z})}^H(t_0z)= p_{t_0z}\cdot t_0z= H(p_{t_0z})+ L(t_0z)=H(p_0)+ L(t_0z).$$
On the other hand, we have that
$$C_{H(p_{t_0z})}^H(t_0z)=t_0C_{H(p_{t_0z})}^H(z)
=t_0C_{H(p_0)}^H(z)=t_0 \langle p_0, z\rangle=\langle p_0, t_0z\rangle.$$
Therefore we obtain
$$p_0\cdot  t_0z=H(p_0)+ L(t_0z),$$
which, together with Lemma \ref{LEM3.1}, implies $t_0z\in\partial H(p_0)$.

Now we return to prove \eqref{union}.
Observe that for any $k>0$, there exist $0<t_1<t_2<\infty$ such that
 $b_s<k<a_t$ for all $t\ge t_2$ and $s<t_1$. To see this,
 let $p_t \in \partial L(tz)$  for $t>0$. Then by Lemma \ref{LEM3.1} (i),
 we have that
 $$ H(p_{t })+  L(t z)= p_t\cdot  t z,$$
 which implies
 $$|p_t|\ge \frac1{t|z|}L(t z)\ge M(tz) \to\fz, \ {\rm{as}}\ t\to\fz.$$
 On the other hand, by Lemma \ref{LEM3.1} (iv), we have that
 $$C=\sup_{0\le t\le 1}\sup_{p\in\partial H(tz)}|p|<\fz,
 \ {\rm{and}}\ H(p_{t })\le Ct|z|\to0,  \ {\rm{as}}\  t\to0.$$
 Define $t(z)=\sup\{t>0: \ b_t< k\big\}.$
 Then  we claim that
 \begin{equation}\label{interval}
 a_{t(z)}\le k\le  b_{t(z)}, \ {\mbox{or equivalently}}\ k\in I_{t(z)}.
 \end{equation}
 In fact, the definition of $t(z)$ implies that for any $\ez>0$,
 $b_{t(z)+\ez}\ge k$. Let  $p_{t(z)+\ez}\in\partial L((t(z)+\ez)z)$
 be such that $H(p_{t(z)+\ez})= b_{t(z)+\ez}$.
Applying Lemma \eqref{LEM3.1} (iv), there exists $p_*\in\partial L(t(z)z)$
such that $p_{t(z)+\ez}\to p_*$ as $\ez\to 0$ and hence
$H(p_*)=\lim_{\ez\rightarrow 0} b_{t(z)+ez}\ge k$.
Then $b_{t(z) }\ge H(p_*)\ge k$. On the other hand,
since  $a_{t(z)-\ez}\le b_{t(z)-\ez}\le k$, a similar argument
shows $a_{t(z) }\le k$.
This completes the proof of Lemma \ref{la-mu}.
\end{proof}

\begin{cor}\label{xew4} For $n\ge 2$, assume  $H$ satisfies both \eqref{assh0}    and ({\bf A}) of Theorem \ref{lla}. For any $R\ge 1$, there exists a constant $C_R>0$ such that
for any $\dz\in (0,1)$ and $0\le k\le R $,  it holds that
\begin{equation}\label{xew5}
C^H_{k}(x)+\dz|x|\le C^H_{k+C_R\dz}(x),\quad \forall x\in \rn.
\end{equation}
\end{cor}

\begin{proof}
By the homogeneity, it suffices to show \eqref{xew5} for $x\in\rn$
with $|x|=1$.  Let $p_x\in H^{-1}(k)$ be such that
$$C_k^H(x)=\langle p_x, x\rangle.$$
Then Lemma \ref{la-mu} implies that there exists $t_x>0$ such that
$t_x x\in\partial H(p_x)$. Moreover, by Lemma \ref{LEM3.1} (iv),
we see that there exists $C(R)>0$ such that
$$|t_x|\le C(R).$$
Note that
$$ C_{k}^H( x)+\dz=p_x\cdot x+\dz= (p_x+\dz x)\cdot  x,$$
and  the convexity of $H$ implies that
$$H(p_x+\dz x)\le  H( {p_x} ) +t_x x \cdot  \dz x=k+t_x\dz\le k+C(R)\dz.$$
Hence \eqref{xew5} holds.
\end{proof}

\begin{cor}\label{com-linear} For $n\ge 2$,
assume $H$ satisfies both \eqref{assh0} and ({\bf A}) of Theorem \ref{lla}.
Then every linear function $u$ is an absolute minimizer for $H$ in $\rn$ .
\end{cor}

\begin{proof} Write $u(x)=a+e\cdot x$ for all $x\in\rn$.  It is obvious
that $u$ is an absolute minimizer when $e=0$. So we may assume $e\ne 0$
so that for  $q\in\partial H(e)$, $q\ne 0$.
Given any domain $\Omega\Subset\rn$, let $v\in C^{0,1}(\Omega)\cap C^0(\overline\Omega)$ with  $v=u$ on $\partial \Omega$.
For $x^0\in\Omega$, denote by $(x,y)$ the component of $\rr q\cap \Omega$ containing $x^0$. We may assume that $y=x+t_0q$ for some $t_0> 0$.  Then
by Lemma \ref{la-mu}, we have
$$v(y)-v(x)= u(y)-u(x)= e\cdot (y-x)= t_0 e\cdot q=t_0C^H_{H(e)}(q).$$
On the other hand, if we let $k=\|H(Dv)\|_{L^\fz(\Omega)}$, then by Lemma \ref{LEM7.13}, we have
$$v(y)-v(x)\le C^H_{k}(y-x)= t_0C^H_{k}(q) .$$
Therefore we obtain that
$$C^H_{H(e)}(q)\le C^H_{k}(q),$$
which implies that $\|H(Du)\|_{L^\infty(\Omega)}=H(e)\le k$, that is, $u\in AM_H(\Omega)$.
\end{proof}

\begin{lem}\label{LEMest} For $n\ge 2$, assume $H$ satisfies \eqref{assh0}.
For $x\in\rn$, $r>0$, and $0<\dz<1$, if $u\in CC_H(B(x_0,r))$
and $e\in \mathscr D(u)(x_0;r; \dz)$, with  $\frac14\le H(e) \le 4$,
then there exists $C>0$ such that
 $$  \|Su\|_{L^\fz(B(x_0,\frac{3r}4))}\le H(e)+C\dz.$$
\end{lem}
\begin{proof} From $e\in \mathscr D(u)(x_0;r; \dz)$, we can deduce that
  $$
 |u(y)-u(x)-e \cdot (y-x) | \le 2\dz r\quad\forall x,y\in B(x_0,r).
$$
Thus by \eqref{xew5} there exists $C>0$
such that for $x\in B(x_0,\frac{3r}4)$ and $y\in\partial B(x,\frac{r}4)$, it holds that
$$u(y)\le u(x)+ e \cdot (y-x)+2\dz r\le u(x)+ e \cdot (y-x)+8\dz|x-y|\le u(x)+ C_{H(e)+C\dz}^H(y-x).$$
This and $u\in CC_H(B(x_0,r))$ imply that
$$u(y)\le    u(x)+ C_{H(e)+C\dz}^H(y-x)\quad \forall y\in  B(x,\frac{r}4).$$
Therefore we obtain
 $$Su(x)\le H(e)+C\dz, \ \forall x\in B(x_0,\frac{3r}4).$$
 This completes the proof.
\end{proof}

\begin{lem}\label{xxlem}
For $n\ge 2$, assume $H$ satisfies both \eqref{assh0} and
({\bf A}) of Theorem \ref{lla}.
For any $\ez>0$, $p\in\rn$  and $R\ge 1$, there exists $\tau (p,R,\ez)$
such that for any $\dz\in(0,\tau )$ and vector $ q\in {\mathbb S}^{n-1} $
 satisfying
\begin{equation}\label{cond}
C^H_{H(p)-R\dz}(q )\le p\cdot q +R\dz ,
\end{equation}
we  have
\begin{equation}\label{est}
\inf_{\hat q\in \partial H(p)}|q-\frac{\hat q}{|\hat q|}|\le \ez.
\end{equation}
\end{lem}

\begin{proof}
First we claim that
\begin{equation}\label{cond1}
C^H_{H(p) }(q )\le p\cdot q+C_1(H,R,p)\dz.
\end{equation}
To see this, let $p_0\in \rn$, with $H(p_0)=H(p)$, and $t_0>0$ be such that
$$C^H_{H(p) }(q )=p_0\cdot q.$$
Then Lemma \ref{xew4} implies that $t_0q\in\partial H(p_0)$.
Let $\theta_0\in(0,1)$ be such that
$H((1-\theta_0)p_0)= H(p)-R\dz$.
Since $H((1-\theta_0)p_0)\le (1-\theta_0)H(p_0)$,
it follows that $\theta_0\le \frac{R\dz}{H(p_0)}$.
 Thus by \eqref{cond} we have
 \begin{eqnarray*}C^H_{H(p) }(q )&=&(1-\theta_0)p_0\cdot q+\theta_0 p_0\cdot q\le
 C^H_{H(p)-K\dz}(q)+ \theta |p_0|\\
 &\le&  p\cdot q +\big(R+R\frac{|p_0|}{H(p_0)}\big)\dz.
 \end{eqnarray*}
This yields \eqref{cond1} with $C_1=R+R\frac{|p_0|}{H(p_0)}$.

It suffices to prove \eqref{est} under the assumption \eqref{cond1}.
We prove this by contradiction.  For, otherwise,
there would exist $\ez_0>0$ and  a sequence $ q_i\in \mathbb S^{n-1}$
satisfying
$$
C^H_{H(p) }(q_i )\le p\cdot q_i +\frac{C_1}{i},
$$
but
\begin{equation}\label{gap}
\inf_{\hat q\in \partial H(p)}|q_i-\frac{\hat q}{|\hat q|}|\ge \ez_0.
\end{equation}
 Assume that   $q_i\to q_\fz\in\mathbb S^{n-1}$ as $i\to\fz$.
Then
$$C^H_{H(p ) }(q_\fz )= p \cdot q_\fz,$$
so that by Lemma \ref{la-mu}, there exists $t_\fz>0$ such that
 $t_\fz q_\fz\in\partial H(p )$. From \eqref{gap}, we would have
 $$|q_i-q_\fz|=\big|q_i-\frac{t_\fz q_\fz}{|t_\fz q_\fz|}\big|
 \ge \inf_{\hat q\in \partial H(p)}|q_i-\frac{\hat q}{|\hat q|}|\ge \ez_0.$$
 This is impossible.
\end{proof}

   \begin{rem}\rm In general, under the assumptions \eqref{assh0} and
   ({\bf A}) of Theorem \ref{lla}, we may not be able to replace the dependence of $\tau$ on
   $p$ by that on $H(p)$. Here is an example. Let $K\subset\mathbb R^2$
   be a symmetric, strictly convex domain, containing $0$, whose
   bounded by a closed  curve $\gz:\mathbb S^1\to\rr^2$ such that
    $$|\gz(t)-\gz(t')|=\ell(\gz)|t-t'|,  \mbox{whenever $ t $ and $ t' $ are sufficiently close}.$$
    Assume that $\gz$ is differentiable, except at $(1,0)$.
    Define $H:\rr^2\to\mathbb R$ by  letting $H(0)=0$,
    $H(p)=1$ whenever $p\in \gz$, and
    $H(p)=|p|^2H(\frac p{|p|})$ for $p\not=0$.
    Then $\partial H(\gz(1,0))=[a,b]$ for some $a,b\ne 0$, with
    $a/|a|\ne b/|b|$. Let $q=\frac{a+b}{|a+b|}$. Then,
    for $\theta\in(0,\frac{\pi}4)$, we have
   $$  C^H_1(q )=\gz(1,0)\cdot q\le  \gz(\cos\theta,\sin\theta)\cdot q
        + |\gz(1,0)- \gz(\cos\theta,\sin\theta)||q |\le  \gz(1,0)\cdot q +C\theta.$$
On the other hand, observe that  $\gz(\cos\theta,\sin\theta)\to \gz(1,0)$ and
      $DH(\gz(\cos\theta,\sin\theta) )$ converges $a$ or $b$ as $\theta\to0^+$.
 Set
$$
q_\theta=\frac{DH(\gz(\cos\theta,\sin\theta) )}{|DH(\gz(\cos\theta,\sin\theta) )|}.
$$
Then we have
$$
\displaystyle\liminf_{\theta\to0^+}|q_\theta-q|\ge \frac12 \big|\frac{ a}{|a|}-\frac{ b}{|b|} \big|>0.
$$
\end{rem}

\section{Sharpness of condition ({\bf A}): a counterexample}

In this section, we will illustrate, by constructing a counterexample, that ({\bf A}) is
sharp.  By the reason as in Section 1.1, we always assume
$H$ satisfies \eqref{assh0}. 
We begin with  the following lemma, which
is motivated by the example constructed by \cite{k11}. 
 \begin{lem} \label{LEM3.4} For $n\ge 2$, assume
$H$ satisfies \eqref{assh0} and $H$ is a constant in  a line segment
$[a,b]\subset\rn$, with $a\ne b$.
 For any   $f\in C^{0,1}\left(\rr\right)$, with $\|f'\|_{L^\fz\left(\rr\right)}\le  1$,
define
\begin{equation}\label{E0}
u_f\left(x\right)=\frac{b+a}2\cdot x +f\big( \frac{b-a}2\cdot x\big)\quad\forall x\in\rn.
\end{equation}
Then  $u_f\in AM_H(\rn)$,  $ H(Du_f) = H(a)$ almost everywhere in $\rn$, and
$Su\equiv H(a)$ in $\rn$.

  \end{lem}

 \begin{proof} It is easy to see that  $u_f\in C^{0,1}(\rn)$ and
     $$Du_f\left(x\right)= \frac{b+a}2+ \frac{b-a}2 f'\big( \frac{b-a}2\cdot x\big),\quad\mbox{almost every $x\in\rn$}.$$
 It follows from $\|f'\|_{L^\fz\left(\rr\right)}\le  1$ that
 $|f'\big(\frac{b-a}2\cdot x\big)|\le 1$ for almost all $x\in\rn$ and
 hence $Du_f\left(x\right)\in[a,b]$.
Thus $H (Du_f\left(x\right) )=H(a)$  for almost all $x\in\rn$.

 To show $u_f\in AM_{H}(\rn)$, let $V\Subset\rn$ be an arbitrary domain
 and $v\in   C^{0,1}\big(\overline V\big)$ be such
 that $v=u_f$ on $\partial V$, we want to show
 $$\|H\left(Dv\right)\|_{L^\fz\left(V\right)}:=s \ge \|H\left(Du_f\right)\|_{L^\fz\left(V\right)}=H(a).$$
This is trivially true, if  $k=H(a)=0$.
So we may assume  $k>0$.
Let $q_0\in \partial H\big(\frac{a+b}2\big)$. Then $q_0\ne0$.
For, otherwise,
$$H\left(0\right)- H\big(\frac{a+b}2\big)\ge -q_0\cdot \frac{a+b}2=0,$$
which implies that $H(a)=H\big(\frac{a+b}2\big)=0$, and contradicts to
$k>0$.

By Lemma \ref{LEMconst},
 we have $q_0\perp \left(b-a\right)$.
Let $x_0\in V$, $\ell_0 =x_0+\rr q_0$, and $\gz_0$
be  the component (an open interval) of $\ell_0\cap V$, containing $x_0$, and $t_0>0$ be  the length of $\gz_0$.
Write $\gz_0 =\left(x_0,y_0\right)$, with $x_0,y_0\in\partial V$, and $y_0=x_0+t_0\frac{q_0}{|q_0|}$. From  $q_0\perp \left(b-a\right)$,
we have that
 $$\frac{b-a}2\cdot y_0=\frac{b-a}2\cdot x_0.$$
Hence we obtain that
\begin{eqnarray*}
 u_f\left(y_0\right)-u_f\left(x_0\right)&=& \frac{b+a}2\cdot \left(y_0-x_0\right)+f\big(\frac{b-a}2\cdot y_0 \big)- f\big(\frac{b-a}2\cdot x_0\big)\\
 &=&\frac{b+a}2\cdot \left(y_0-x_0\right)=\frac {t_0}{|q_0|}\frac{b+a}2\cdot q_0.
\end{eqnarray*}
From $q_0\in \partial H\big(\frac{a+b}2\big)$ and Lemma \ref{LEM3.1}, this yields that
\begin{equation}\label{E25}u_f\left(y_0\right)-u_f\left(x_0\right)= \frac {t_0}{|q_0|}
\big(H\big(\frac{b+a}2\big)+L\big(q_0\big)\big)= \frac {t_0}{|q_0|} \big(k+L(q_0)\big).
\end{equation}

On the other hand,  it follows from Lemma \ref{LEM7.11} that
  $$v\left(y\right)-v\left(x\right)\le C_s^H\left(y-x\right),\ \mbox{whenever $[x,y]\subset V$}.$$
Hence we have
$$ v\left(y_0\right)-v\left(x_0\right)\le C_s^H\left(y_0-x_0\right)
=\frac {t_0}{|q_0|}\sup_{H(p)\le s} p\cdot q_0\le
\frac {t_0}{|q_0|} \big(s+L(q_0)\big).$$
From $u\left(y_0\right)-u\left(x_0\right)=v\left(y_0\right)-v\left(x_0\right)$,
and \eqref{E25}, we then conclude that
 $$\frac {t_0}{|q_0|} \big(s+L(q_0)\big)
 \ge \frac {t_0}{|q_0|} \big(H(a)+L(q_0)\big).$$
This yields that $s\ge H(a)$, i.e.
$\|H\left(Dv\right)\|_{L^\fz\left(V\right)} \ge \|H\left(Du_f\right)\|_{L^\fz\left(V\right)}$.
This completes the proof of Lemma \ref{LEM3.4}.
\end{proof}

The following result indicates that   ({\bf B}), ({\bf C})  or  ({\bf D}) may fail, without
condition ({\bf A}) in Theorem \ref{lla}.

\begin{lem}\label{nolla} For $n\ge 2$, assume
 $H$ satisfies \eqref{assh0} and $H$ is a constant in $[a,b]\subset\rn$,
 with $a\ne b$. Let $f(t)=|t|$ for $t\in\rr$ and
 $u_f$ be given by \eqref{E0}.
 Then $u_f$ is neither $C^1$, nor does not enjoy the linear approximation property
 \eqref{LAP}. If $\Omega=\rn$, then $u_f$ does not satisfy the Liouville property.
 \end{lem}
 \begin{proof}
 By choosing a new coordinate system, we may assume that
$a=0$ and $b=\lambda_0 e_n$ for some $\lambda_0>0$, where
$e_n=(0',1)$. Then
$$
u_f(x)=\frac{\lambda_0}{2}(x_n+|x_n|)
=\begin{cases} 0 & x_n<0\\
\lambda_0x_n & x_n\ge 0.
\end{cases}
$$
It is easy to see that $u_f$ is neither differentiable nor can be
linearly approximated at $(x',0)$ for any $x'\in\mathbb R^{n-1}$.
Also note that $\|H(Du_f)\|_{L^\fz(\rn)}=\lambda_0$ so  that
$\| Du_f \|_{L^\fz(\rn)}<\fz$ and $u_f$ has a linear growth near
the infinity. However, $u_f$ is not a linear function.
\end{proof}

\section{Proofs of Theorem \ref{LEMC4} and Theorem \ref{lla}:  ({\bf A})$\Rightarrow$({\bf B})}

This section is devoted to the proof of Theorem \ref{lla} ({\bf A}) $\Rightarrow$ ({\bf B}).
A crucial ingredient of the proof is the following theorem.

\begin{thm}\label{LEMC4}  For $n\ge 2$,
assume $H$ satisfies \eqref{assh0}  and ({\bf A}) of Theorem \ref{lla}.
If $u\in C^{0,1}(\rn)$   satisfies
\begin{equation}\label{F1}\mbox{$S^+_tu\left(0\right)=-S^-_tu\left(0\right)=k$,
\ {\rm{and}}\ $\max\big\{S^+_tu\left(x\right), -S^-_tu\left(x\right)\big\}\le k $,
$\forall x\in\rn$, $t>0$   }
  \end{equation}
  for some $0\le k<\fz$, then  there exists a vector $p_0\in\rn$ such that
  $H(p_0)=k$ and
 $$\mbox{$u(x)=u(0)+p_0\cdot x,$\ $\forall  x\in\rn$.}$$
  \end{thm}

Employing  Lemmas \ref{LEM3.1} and \ref{LEMC3}, we first establish a
 weaker version of Theorem \ref{LEMC4}. Namely,

\begin{lem}\label{LEMC5}  For $n\ge 2$,
assume $H$ satisfies \eqref{assh0}  and ({\bf A}) of Theorem \ref{lla}.
If $u\in C^{0,1}(\rn)$  satisfies  \eqref{F1} for some $0< k<\fz$,  and
\begin{equation}\label{F2}
\mbox{$u(e)=k+L(e)$, and $u(se)=su(e),\ \forall s\in\rr,$}
\end{equation}
for some vector $0\ne e\in\rn$, then
there exists a vector $p_0\in\partial L(e)$ such that $H(p_0)=k$
and
  $$\mbox{$u(x)=  p_0\cdot x$,\   $\forall  x\in\rn$.}$$
  \end{lem}

  \begin{proof} (see Figure 1 below). Observe that by \eqref{F2}  $u(0)=0$ , and by \eqref{F1}
  $S^+_{ t}u\left(  (t-s)e\right)\le  k  $ for all $t>0$ and $s\in\rr$.
  Hence we have
  $$
   \frac{u\left(x+se\right) -u\left(   \left(s-t\right)e\right) }t
   -L\big(\frac{x+se-(s-t)e}t\big)\le  k \quad\forall x\in\rn.$$
   This, combined with  \eqref{F2} again,  implies that
   $$
   -\frac{u\left(x+se\right) }t +\frac{  s}t u\left(   e\right)\ge u\left(e\right) - k   -L\big( e+\frac{  x }t\big)=L(e) -L(e+\frac{  x }t), \ \forall x\in\rn, t>0,s\in\rr.$$
 Hence by the convexity of $L$, there exists $p_{t,x}\in
 \partial L\big(e+\frac{x}{t}\big)$ such that
 $$
  u(x+se) - su(e)\le  p_{t,x}\cdot   x, \quad\forall x\in\rn, t>0,s\in\rr.$$
 By Lemma \ref{LEM3.1} (iv), there exists $p_x\in \partial L\left(e\right)$
 such that $p_{t,x}\to p_x$ as  $t\to\fz$.
 Therefore, we obtain that
\begin{equation}\label{F3}
    u(x+se) - su(e)\le  p_{x}\cdot   x, \quad\forall x\in\rn,  s\in\rr.
 \end{equation}

Similarly, from $- S^-_{ t}u\big(  -(t-s)e\big)\le  k  $ for  $t>0$ and $s\in\rr$,
we can conclude that there exists $\hat p_x\in\partial L\left(e\right)$ such that
 \begin{equation}\label{F4}
    u(x+se) - su(  e)\ge  \hat p_{x}\cdot   x, \quad\forall x\in\rn,
     s\in\rr.
  \end{equation}

 If $\partial L\left(e\right)=\{p_0\}$ is a singleton, then we have
  $p_x=\hat p_x=p_0$ for all $x\in\rn$ so that \eqref{F3} and \eqref{F4}
  imply $ u\left(x\right)=p_0\cdot x$ as desired.

If $\partial L\left(e\right)$ is not a singleton, then by
Proposition \ref{LEMC3}, $\partial L\left(e\right)$ must be a line segment
$[a,b]$ in  a line $\ell\subset\rn$. Therefore, $p_{e}\in [a,b]$ and $\hat p_{e}\in [a,b]$.
Applying \eqref{F3} and \eqref{F4} with $x=e$ and $s=0$,  we have that
$$\hat p_{e}\cdot   {e} \le u\left(e\right)\le   p_{e}\cdot   {e}.$$
Thus there exists $\theta_0\in [0,1]$ such that $ {p}_\ast:=\theta p_e
+(1-\theta)\hat{p}_e\in [a,b]=\partial L(e)$, and
$$ {p}_\ast\cdot e=u(e).$$
This, together with \eqref{F2} and Lemma \ref{LEM3.1} (i), implies that
$$ k =L\left(e\right)-u\left(e\right)= L\left(e\right)- {p}_\ast\cdot { e}=
H\left( {p}_\ast\right).$$

Now we want to show that
$$u(x)= {p}_\ast\cdot x,\ \forall x\in\rn.$$

\begin{center}
\begin{tikzpicture}
\fill (0,0) circle (1.5pt);
\draw (0,0) node [right]{$O$};
\draw [->,>=latex, thick, color=blue](0,0)--(0,1.9);
\draw (0,1.9) node [right]{$z$};
\draw [->,>=latex,thick](0,0)--(-1.5,3);
\draw (-1.5,3) node [right] {$x$};
\draw [dotted,thick](0,1.9)--(-1.5,3);

\draw [dotted,thick](-1.3,0.94)--(-1.5,3);

\draw [thick](-3,2.25)--(0,0);
\draw [dotted, thick](0,0)--(2,-1.5);
\draw [ thick](2,-1.5)--(3,-2.25);
\draw (-3,2.25) node [above]{$\rr e$};
\draw [ thick, color=red] (-2,0.3)--(-0.5,-1);
\draw (-0.4,-1) node [right]{$\ell$};

\draw [->,>=latex, thick](0,0)--(-1.5,-0.13);
\draw (-1.5,-0.13) node [left]{$\hat p_z$};
\draw [->,>=latex, thick](0,0)--(-1.2,-0.39);
\draw (-1.1,-0.48) node [left]{$p_\ast$};
\draw [->,>=latex, thick](0,0)--(-0.7,-0.83);
\draw (-0.7,-0.83) node [left]{$p_z$};

\draw [dotted, thick](-2.60,1.5)--(5,1.5);
\draw (5,1.5)--(2.60,-1.5)--(-5,-1.5)--(-2.60,1.5);
\draw (0,-1.8) node [below] {${\rm Figure \ 1}$};
\end{tikzpicture}
\end{center}
First, observe that
\begin{equation}\label{pos}
\left(p_\ast-p \right)\cdot e \ne 0, \forall p \in \partial L\left(e\right)\setminus \{  p_\ast\}.
\end{equation}
For, otherwise, by Proposition \ref{LEMC3}, there exists
$p_1\in \partial L(e)\setminus\{p_\ast\}$ such that
$p_1\cdot e= p_\ast\cdot e$,  and hence
$$H\left(  p_\ast\right)=  p_\ast\cdot e- L\left(e\right)= p_1\cdot e-L\left(e\right) =H\left(p_1\right).$$
Hence $H$ is constant on the line segment $[p_\ast, p_1]$, which
is impossible.

From \eqref{pos}, we see that
\begin{equation}\label{F5}
\rn=\bigcup_{\{z\in\rn: z\perp \ell- p_0\ast\}} \Big(\{z\}+\rr e\Big).
\end{equation}
For any $x\in\rn$, there exist $s\in\rn$,
and $z\in\rn$ that is perpendicular to $\ell-{   p_\ast}$,
such that $x=z+se$. By \eqref{F3} and \eqref{F4},
there are $p_z, \hat p_z\in \partial L\left(e\right)\subset\ell$ such that
$$\hat p_z\cdot z\le u\left(z+s e\right)-u\left(s e \right)\le p_z\cdot z.$$
Since $p_z, \hat{p}_z\perp\ell\setminus\{p_\ast\}$, we have
$$\hat p_z\cdot z=   p_\ast  \cdot z= p_z\cdot z.$$
Hence
$$   p_\ast \cdot z\le u\left(z+s e\right)-u\left(s e \right)\le   p_\ast \cdot z.$$
This implies
$$u(x)=u \left(z+s e \right)=u\left(s e \right)+   p_\ast \cdot z
= s u(e)+  p_\ast\cdot z= s   p_\ast \cdot  e +   p_\ast \cdot z
=  p_\ast \cdot \left(z+s e \right)=  p_\ast\cdot x.
$$
This completes the proof. \end{proof}

With the help of Lemmas \ref{LEM3.1}, \ref{LEMC3} and \ref{LEMC5},
we are ready to prove Theorem \ref{LEMC4}.

 \begin{proof}[Proof of Theorem \ref{LEMC4}.]
 Without  loss of generality, assume $u\left(0\right)=0$.
 If $ k =0$, then by \eqref{F1} we have that $Su(x)=0$ for
 all $x\in\rn$ and hence $H(Du)\equiv 0$ on $\rn$. This
 implies that $Du\equiv 0$ and hence $u\equiv u(0)=0$.

 Assume $ k >0$ below. First, we have

 \medskip
\noindent {\it Claim A.  There exists $0\ne y^\pm\in\rn$ such that
 \begin{equation}\label{E6}
  u\left(  y  ^+ \right) - L\left(  y ^+  \right) =     k  \quad \mbox{and   }\quad u\left(s y ^+ \right) =s u\left(y^+ \right) >0  \quad\forall s>0,
 \end{equation}
\begin{equation}\label{E7}
  -u\left(  y  ^- \right) - L\left(  -y ^-  \right) =     k  \quad \mbox{and   }\quad u\left(s y ^- \right) =s u\left(y^- \right)< 0\quad\forall s>0,
 \end{equation}
  \begin{equation}\label{E8} L\left( \lz  y ^+  +\left(1-\lz\right)(-y^-)\right)=
  \lz L\left(y^+\right)  + \left(1-\lz\right) L\left(-y^-\right) 
    \quad\forall \lz\in(0,1)  \end{equation}
and  }
\begin{equation}\label{E10} u\left(\lz t y^++\left(1-\lz\right)s y^-\right)= \lz t u\left( y^+\right)+\left(1-\lz\right)s u\left(y^-\right)
\quad\forall \lz\in(0,1), \ t,s>0.\end{equation}

\begin{proof}[Proof of Claim A ]
 It follows from  $S^+_tu\left(0\right)= k =-S^-_tu\left(0\right)$ for all $t>0$, and Lemma \ref{LEM3.3} that there exist $R_k>0$ and
 $y^\pm_t\in \overline {B\left(0,R_k t\right)}$ such that
\begin{equation}\label{E1}
\frac{u\left(y_t^+\right)}t - L\big(\frac{y_t^+}t\big)
= k =-\frac{u\left(y_t^-\right)}t - L\big(\frac{-y_t^-}t\big),\quad \forall t>0,
\end{equation}
and hence
\begin{equation}\label{E2}
   \frac{u\left(y_t^+\right) }{ 2t}-  \frac{ u\left(y_t^-\right)}{ 2t} =   k  +\frac12
   \Big(L\big(\frac{y_t^+}t\big)  + L\big(\frac{-y_t^-}t\big)\Big),
  \quad \forall t>0.
    \end{equation}
Since $|\frac{y^\pm_t}{t}|\le R_k$, there exists $y^\pm\in \rn$
such that after passing to a subsequence, $\frac{y^\pm_t}{t}\to y^\pm$
as $t\to\fz$.

To show \eqref{E8},  observe that by
$S^+_{2t}u\left(y_t^-\right)\le  k  $ for all $t>0$, we have
$$
  \frac{u\left(y_t^+\right)-u\left(y_t^-\right)}{2t} - L\big(\frac{y_t^+-y_t^-}{2t}\big)\le  k  \quad \forall t>0. $$
This and \eqref{E2}   yield
$$
\frac12\Big(L\big(\frac{y_t^+}t\big)  +L\big(\frac{-y_t^-}t\big)\Big)
\le L\big(\frac{y_t^+-y_t^-}{2t}\big) \quad \forall t>0, $$
which, together with the convexity of $L$, yields that
$$
  \frac12\Big(L\big(\frac{y_t^+}t\big)  + L\big(\frac{-y_t^-}t\big)\Big)
  = L\big(\frac{y_t^+-y_t^-}{2t}\big), \quad \forall t>0. $$
Applying the convexity of $L$ again,  we see that
  $L$ must be  linear in $\big[-\frac{y^-_t}{t},\frac{y^+_t}{t}\big]$, that is,
   \begin{equation}\label{E3}
   L\big( \lz \frac{y_t^+}t+\left(1-\lz\right)\frac{-y_t^-}t \big)=
  \lz L\big(\frac{y_t^+}t\big)  + \left(1-\lz\right) L\big(\frac{-y_t^-}t\big),
  \quad \forall t>0,\ \lz\in\left(0,1\right).
  \end{equation}
This, after sending $t$ to $\fz$, implies
 \eqref{E8}.

To show \eqref{E6} and \eqref{E7},
observe that $S^+_{\left(\theta_2-\theta_1\right) t}u\left(\theta_1 y_t^+\right) \le  k  $ for $0\le \theta_1<\theta_2\le1$ and $t>0$. Hence
$$
\frac{u\left(\theta_2 y_t^+\right) -u\left(\theta_1 y_t^+\right)}{\left(\theta_2-\theta_1\right) t} - L\big(\frac{y_t^+}t\big)\le  k,  $$
  that is,
  $$u\left(\theta_2 y_t^+\right) -u\left(\theta_1 y_t^+\right)
  \le \Big( k + L\big(\frac{y_t^+}t\big)  \Big)\left(\theta_2 t-\theta_1t\right).$$
This, combined  with \eqref{E1}, yields that
  $$
 u\left(\theta_2 y_t^+\right)-u\left(\theta_1 y_t^+\right)=
 \Big(k + L\big(\frac{y_t^+}t\big) \Big)
\left(\theta_2 t-\theta_1t\right)\quad\forall t>0, \ 0\le \theta_1<\theta_2\le1. $$
 In particular, for all $0\le s\le t$, choosing $\theta_1=0$ and
 $\theta_2=\frac{s}{t}$ and applying \eqref{E1} again,  we obtain that
\begin{equation}\label{E4}
 \frac{1}{s}u\big( s \frac{y_t^+}{t}\big)- L\big( \frac{y_t^+}t \big)
 =     k  \quad
 \mbox{and  } \quad u\big(\frac{s y_t^+}t\big)
 =s u\big(\frac{  y_t^+}t\big).
 \end{equation}
Similarly, we also have that for all $0\le s\le t$,
\begin{equation}\label{E5}
 -\frac{1}{s}u\big(s \frac{y_t^-}{t}\big)- L\big( -\frac{y_t^-}t \big) =     k  \quad \mbox{and   }\quad u\big(\frac{s y_t^-}t\big) =s u\big(\frac{  y_t^-}t\big).
 \end{equation}
It is clear that \eqref{E6} and \eqref{E7}
follow from \eqref{E4} and \eqref{E5}
by sending $t$ to $\fz$.


Now we want to prove \eqref{E10}.
By $S^+_{t\lz} u\left(\left(1-\lz\right)sy^-\right)\le  k $ for all $t,s>0$ and $\lz\in(0, 1)$, we have that
 $$\frac{u\left(\lz ty^++\left(1-\lz\right)sy^-\right)-u\left(\left(1-\lz\right)sy^-\right)}{t\lz} -L\left(y^+\right)\le  k.$$
This, together with \eqref{E6} and \eqref{E7}, yields that
\begin{equation}\label{E9} u\left(\lz t y^++\left(1-\lz\right)s y^-\right)\le \lz t u\left( y^+\right)+ \left(1-\lz\right)s u\left(y^-\right)\quad\forall t,s>0,
\ \forall \lz\in (0,1).
\end{equation}
Similarly,
  by $-S^-_{\left(1-\lz\right)s}u\left(\lz t y^+\right)\le  k $ for $t,s>0$
  and $\lz\in (0,1)$, we have that
 $$-\frac{u\left(\lz t y^++\left(1-\lz\right)s y^-\right)-u\left(\lz t y^+\right)}{\left( 1-\lz\right)s}-L\left(-y^-\right)\le  k ,$$
which together with \eqref{E6} and \eqref{E7} yields again
\begin{equation}\label{E11x}
 u\left(\lz ty^++\left(1-\lz\right)s y^-\right)\ge \lz tu\left( y^+\right)+\left(1-\lz\right)su\left(y^-\right)\quad\forall t,s>0, \ \forall\lz\in (0,1).
 \end{equation}
It is clear that \eqref{E10} follows from \eqref{E9} and \eqref{E11x}.
Hence {\it Claim A} is proved.
\end{proof}

\medskip
Observe that by \eqref{E6}, \eqref{E7} and \eqref{E10},
there exists  $\lz_0\in\left(0,1\right)$ such that for any $s>0$,
 \begin{equation}\label{E11} u\left(\lz_0 sy^++\left(1-\lz_0\right)sy^-\right)=  s\big(\lz_0   u\left( y^+\right)+\left(1-\lz_0\right) u\left(y^-\right)\big)=0.
 \end{equation}
We proceed with two cases:
\begin{enumerate}
\item[{\it Case 1.}] $\lz_0 y^++\left(1-\lz_0\right)y^-=0$.
In this case, we have
  $-y^-=  s_0y^+$ and $u\left(y^-\right)= -s_0u\left(y^+\right)$,
  with $s_0=\frac{\lz_0}{1-\lz_0} $.  Hence
  $$u\left(-  y^+\right)= u\big(\frac{y^-}{s_0}\big)=\frac{1}{s_0}u\left(y^-\right)= - u\left( y^+\right).$$
  This, together with \eqref{E6}, yields
 \begin{equation}\label{E12}
  u\left(  y  ^+ \right) - L\left(  y ^+  \right) =     k  \quad \mbox{and   }\quad u\left(s y ^+ \right) =s u\left(y^+ \right),  \quad\forall s\in\rr.
 \end{equation}
Hence $u$ satisfies the condition \eqref{F1},  with $e=y^+$.
Applying Lemma \ref{LEMC5}, we conclude that
there is a vector $p_0\in \partial L(y^+)$ such that $H(p_0)=k$
and $u(x)=p_0\cdot x$ for all $x\in\rn$.

\smallskip
\item [{\it Case 2}.] $\lz_0 y^++\left(1-\lz_0\right)y^-\ne 0$ (see Figure 2 below).
In this case, $y^\pm$ does not lie in the same line.
For, otherwise, from $u(s\lz_0 y^++s\left(1-\lz_0\right)y^-)=0$ for $s\ge0$
we see that we know that $u\equiv  0$ either in $\rr_+y^+$ or $\rr^+y^-$
so that either $u(y^+)=0$ or $u(y^-)=0$, which contradicts to \eqref{E6}
or \eqref{E7}.  Set $x_s=  \lz_0 sy^+ +\left(1-\lz_0\right)s y^- $, and
 define a function $v_s$ by letting
 $$v_s\left(z\right)=u\left(x_s+z\right),\ z\in\rn.$$
  By \eqref{E11}, we have $v_s\left(0\right)=0$.
  By $u\in C^{0,1}(\rn)$ and the Arzela-Ascoli theorem,
  there exists $v\in C^{0,1}(\rn)$ such that that as $s\to\fz$,
  $v_s\rightarrow v$ locally uniformly in $\rn$.
\end{enumerate}

\begin{center}
\begin{tikzpicture}

\draw (0,0) node [left]{$O$};

\draw [->,>=latex, thick, color=blue](0,0)--(2,1.25);
\draw (2,1.25) node [right]{$y^+$};  

\draw [->,>=latex, thick, color=blue](0,0)--(0.25,-1.25);
\draw (0.25,-1.25) node [right]{$y^-$};

\draw [->,>=latex, thick ](0,0)--(0.875,1.25 );
\draw (0.875,1.25 ) node [right]{$y^0$};

\draw [    ](0.2,-1)--(1.6,1);
\draw (1.3,0.6) node [right]{$u$ is linear};

 \draw [thick, color=red](0,0)--(2.8,0);
 \draw (2.1,0) node [below]{$\{x_s,s>0\}$};
 \draw (0,-1.4) node [below] {${\rm Figure \ 2}$};
\end{tikzpicture}
\end{center}

Writing $y^0=\frac{y^+-y^-}2 $. Then we have the following claim:
\medskip

\noindent {\it Claim B.
There exists  $p_0\in \partial L(y^0)$ such that $H\left(p_0\right)= k $, and
\begin{equation} \label{E-3}
v\left(x\right)=p_0\cdot x,\ \forall x\in\rn.
\end{equation}}
\begin{proof}[Proof of Claim B] By Lemma \ref{LEMC5}, it suffices
to show that $v$ satisfies  both \eqref{F1} and \eqref{F2}, with $e=y^0$.
  To see this, first observe that by \eqref{E10} and \eqref{E11} we have
 \begin{align}\label{E15}
 u\left(x_s+\dz y^0\right)& =  \big(\lz_0 s+\frac\dz 2\big) u\left(y^+\right)
 +\big((1-\lz_0) s-\frac\dz 2\big) u\left(y^-\right)\\
 &
   = \lz_0 s  u\left(y^+\right)+ \left(1-\lz_0\right) s  u\left(y^-\right)+ \frac\dz 2
   \big(u\left(y^+\right)-u\left(y^-\right)\big)\nonumber\\
   &=  \frac\dz 2\big(u\left(y^+\right)-u\left(y^-\right)\big),
   \quad \forall s>0,\ -2\lz_0s<\dz< 2\left(1-\lz_0\right)s.\nonumber
   \end{align}
   Thus
\begin{equation}\label{E16}v\left( \dz y^0\right) = \lim_{s\to\fz}u\left(x_s+\dz y^0\right)=\frac\dz 2\big(u\left(y^+\right)-u\left(y^-\right)\big)
=\dz v(y^0),\quad \forall  \dz\in\rr.
\end{equation}
In particular, by \eqref{E8}
\begin{equation}\label{E17} v(y^0)=\frac1 2\big(u\left(y^+\right)-u\left(y^-\right)\big)=k+\frac12\big(L\left(y^+\right)+L\left(-y^-\right)\big)
=k+ L\big(\frac{y^+ -y^-}2\big)=k+L(y^0).
 \end{equation}
This implies that
\begin{equation}\label{E18} \pm S^\pm_tv\left(0\right)\ge \pm \frac {v\left( \pm ty_0\right)}t -L(y^0)=v (  y^0 )   -L(y^0)=   k, \quad\forall t>0.\end{equation}
One the other hand, for any $x\in\rn$ and $t>0$, it holds that
\begin{align}\label{E19}
\pm S^\pm_tv\left(x\right)&=\sup_{y\in\rn}\Big(
\pm \frac {v\left( \pm y +x\right)- v\left(   x\right)}t -L\big(  \frac y t \big)\Big) \\
&=\sup_{y\in\rn}\lim_{s\to\fz}
\Big(\pm \frac {u\left( \pm y +x+x_s\right)- u\left(   x+x_s\right)}t -L\big(  \frac y t \big)\Big)\nonumber\\
&\le \limsup_{s\to\fz}
\pm S^\pm_tu\left(x+x_s\right)\nonumber\\
&\le k.
\nonumber\end{align}
Combining \eqref{E16}, \eqref{E17}, \eqref{E18} and \eqref{E19}, we see  that $v$ satisfies \eqref{F1} and \eqref{F2}. Thus  \eqref{E-3}
follows from Lemma \ref{LEMC5}.
This proves {\it Claim B}.
\end{proof}

Next we  claim that

 \medskip
{\noindent \it Claim C. There exists vectors $p^\pm_x\in\partial L(\pm y^\pm)$ such that }
 \begin{equation}\label{E22} u\left(x-sy^-\right)+su\left( y^-\right)\le p^-_{ x}\cdot  x\quad\forall x\in\rn, s\in\rr \end{equation}
 and
     \begin{equation}\label{E23}u\left(x-sy^+\right)+su\left( y^+\right)\ge   p^+_{ x}\cdot  x\quad\forall x\in\rn, s\in\rr. \end{equation}

\begin{proof}[Proof of Claim C] Since
$S^+_{t}u\left( \left(t-s\right)y^-\right)\le  k $ for all $t>0$ and $s<t$, one has
       $$ \frac{u\left(x-sy^-\right)-u\left(\left(t-s\right)y^-\right)}t-L\Big(\frac{x-sy^--\left(t-s\right)y^-}{t}\Big)\le  k, $$
      which, together with \eqref{E7} and $t-s>0$, implies that
       $$ \frac{u\left(x-sy^-\right)+su\left( y^-\right)}t\le  k +u\left(y^-\right)+L\big(\frac{x-t y^-}{t}\big)=-L\left(y^-\right)+L\big(-y^-+\frac{x }{t}\big), $$
 so that
 $$ -[u\left(x-sy^-\right)+su\left( y^-\right)]\ge \frac1t\Big(L\left(y^-\right)
 -L\big(-y^-+\frac{x }{t}\big)\Big)\ge -p^-_{t,x}\cdot x,  $$
 for any $p^-_{t,x}\in\partial L(-y^-+\frac{x }{t})$. By Lemma \ref{LEM3.1} (iii),
 there exists $p^-_x\in\partial L\left(-y^-\right)$ such that
 $p^-_{t,x}\to p^-_x$ as $t \to\fz$. Thus \eqref{E22} follows.
Similarly, by $-S^-_{t}u\left( \left(t-s\right)y^+\right)\le  k $ for all $t>0$ and $s<t$,
we can prove \eqref{E23}. This proves {\it Claim C}.
\end{proof}


Finally we will prove Theorem \ref{LEMC4}.
Let $p_0\in\rn$ be given by \eqref{E-3}.  Then by Lemma \ref{LEM3.1} (iv) and \eqref{E8},
$p_0\in \partial L(y^0)$ implies that $p_0\in \partial L \left(\pm y^\pm\right)$.
Hence  by Lemma \ref{LEM3.1} (i), we have
\begin{equation}\label{E20}\pm p_0\cdot y^\pm =H\left(p_0\right)+L\left(\pm y^\pm\right)=  k +L\left(\pm y^\pm\right) =\pm u\left(y^\pm\right).
\end{equation}
Now we divide it into four sub cases.

\medskip

\noindent{\it \underline{Subcase a}. $\partial L\left(y^+\right)\cup \partial L\left(-y^-\right)=\big\{p_0\big\}$.} Then $ p^\pm_{ x}=p_0 $.
Applying \eqref{E22} and \eqref{E23} with $s=0$,
we have that
 $$p_0\cdot x= p^+_{ x}\cdot  x\le  u\left(x \right) \le p^-_{ x}\cdot  x= p_0\cdot x,\quad\forall x\in\rn, $$
hence $u\left(x\right)=p_0\cdot x$ for all $x\in\rn$.

\medskip
\noindent   {\it \underline{Subcase b}.   $\partial L\left(y^+\right) $ contains more than one point; while $ \partial L\left(-y^-\right)=\big\{p_0\big\}$.}
Thus we have that $p_x^-=p_0$. By \eqref{E20} with $s=0$, we then have
$$u\left(x\right)\le p_x^-\cdot x=p_0\cdot x,\quad\forall x\in\rn.$$

By Proposition \ref{LEMC3},  $\partial L\left(y^+\right)=[a,b]$ is a line segment contained in a line $\ell_{y^+}\subset\rn$. As in \eqref{pos} and \eqref{F5},
we know that $y^+$  is not perpendicular to $\ell-p_0$, and
$$\rn=\bigcup_{\{z\in\rn: \ z\perp \ell_{y^+}-p_0\}}\left(\{z\}+\rr y^+\right).$$

For any $z\in\rn$, that is perpendicular to
$\ell_{y^+}\setminus\big\{p_0\big\}$, we have
that $  p^+_x\cdot z=p_0\cdot z $. Hence by \eqref{E23} and \eqref{E20}
we have
$$ u\left(z-sy^+\right)\ge p_x^+\cdot z-su\left( y^+\right)=p_0\cdot  z-sp_0\cdot y^+=p_0(z-sy^+), \quad\forall z\in\rn,\ s\in\rr.$$
Hence we obtain that
$$u\left(x\right)\ge p_0\cdot x,\quad\forall x\in\rn,$$
and hence $ u\left(x\right) = p_0\cdot x$ holds for all $x\in\rn$.

      \medskip
     \noindent    {\it \underline{Subcase c}.  $\partial L\left(-y^-\right) $ contains more than one point; whil $ \partial L\left( y^+\right)$ consists of one point $p_0$.}
This case can be proved exactly in the same way as {\it Subcase b}.

         \medskip
     \noindent    {\it \underline{Subcase d}. Both $\partial L\left(y^+\right) $ and
     $\partial L\left(-y^-\right) $ contain more than one point.}
     By Proposition \ref{LEMC3},
     $\partial L\left(  y^+\right) $  is a line segment contained in the line, say $\ell_{y^+}$, and
      $\partial L\left(-y^-\right) $ is a line segment contained in the line, say $\ell_{y^-}$.
  By an argument similar to Subcase b, we know that
  $y^\pm$  is not perpendicular to $\ell_{y^\pm}-p_0$ and hence
  $$\rn=\bigcup_{\{z\in\rn:
  z\perp \ell_{y^+}- p_0 \}}
  \left(\{z\}+\rr y^+\right)
  =\bigcup_{\{z\in\rn: z\perp \ell_{y^-}- p_0 \}}
  \left(\{z\}+\rr y^-\right).$$
For any $z\in\rn$, with $z\perp \ell_{y^-}\setminus\{p_0\}$,
we have that $ p_0\cdot z=p_x^-\cdot z$.
Thus by \eqref{E22},
$$u\left(z-sy^-\right)+su\left( y^-\right)\le p_0\cdot  z,  \quad\forall    s\in\rr,$$
which, together with  \eqref{E20}, gives
$$u\left(z-sy^-\right) \le p_0\cdot  \left(z-sy^-\right), \quad\forall   s\in\rr.$$
This implies that
$$ u\left(x\right) \le p_0\cdot   x, \ \forall x\in\rn. $$
Similarly,  we can show
$$ u\left(x\right) \ge p_0\cdot   x,  \forall x\in\rn.$$
Hence $ u\left(x\right) = p_0\cdot   x $ for all $x\in\rn$.
The proof of Lemma \ref{LEMC4} is now complete.
\end{proof}


With Theorem \ref{LEMC4} and Lemma \ref{LEM3.4},
we are ready to prove Theorem \ref{lla}: ({\bf A}) $\Rightarrow$({\bf B}).

\begin{proof}[Proof of ({\bf A}) $\Rightarrow$ ({\bf B}) in Theorem \ref{lla}.]
As explained in the introduction,
we may assume $H$ satisfies the assumption  \eqref{assh0}.
Let $u\in AM_H\left(\Omega\right)$ and $x^0 \in \Omega$.
For simplicity, assume $x^0=0$ and $u\left(0\right)=0$.
For a small $\dz_0>0$,  let $U=B\left(0,\dz_0\right)\Subset\Omega$
 and  $K:=\|u\|_{C^{0,1}\left(U\right)}<\fz$. Then $u\in AM_H\left(U\right)$
 is bounded.
  For any $0<r<1$, set $u_r= \frac1r u\left(rx \right) $ for $x\in \frac1rU$.
  Then $u_r\in AM_H\left(\frac1rU\right)$ and
  $\|u_r\|_{C^{0,1}\left( {\frac{1}{r} U)}\right)}=K$.
After passing to a subsequence, we may assume that
there exists $v\in C^{0,1}(\rr^n)$ such that
$u_r \to v $ locally uniformly in $\rr^n$ so that
$$\|v\|_{C^{0,1}(\rr^n)}\le \|u\|_{C^{0,1}\left(\frac{1}{r}U)\right)}=K.$$
It suffice to show that $v$ is a linear function and
$H\left(Dv\right)=S u\left(0\right)$.
To achieve this, we will verify that
$v$ satisfies all the assumptions in Theorem \ref{LEMC4},
with $k=S u\left(0\right)$.

For this purpose, let $R_K>0$ be given by Lemma \ref{LEM3.3}.
From Lemma \ref{LEM3.3}, it holds that for any $x\in\rr^2$,
     \begin{align*}   S^+_tv\left(x\right)&=\sup_{|y-x|\le R_Kt}\frac1t\Big(v\left(y\right)-v\left(x\right)-tL\big(\frac{ y-x}t\big)\Big) \\
     &=\sup_{|y-x|\le R_Kt}\lim_{r\to0}\frac1t\left[\frac{u\left(ry\right)-u\left(rx\right)}r-tL\big(\frac{ y-x}t\big)\right] \\
     &\le\liminf_{r\to0}\sup_{|rx-y|\le R_Ktr}\frac1{tr}\left[ u\left( y\right)-u\left( rx\right) -trL\big(\frac{ y-rx}{ tr }\big)\right]\\
     &\le \liminf_{r\to0} S^+_{tr}u\left(rx\right).
                             \end{align*}
This, combined with Lemma \ref{LEM7.13}, implies that
\begin{equation}
\label{E26}  S^+_tv\left(x\right)\le
\liminf_{r\to0} S^+_{tr}u\left(rx\right)
\le\liminf_{\dz\to0}\liminf_{r\to0} S^+_{\dz}u\left(rx\right)
=  \liminf_{\dz\to0}S^+_\dz u\left(0\right)
= Su\left(0\right)\quad\forall x\in\rn.
                              \end{equation}
                              Similarly, we have  that
                              \begin{equation}\label{E27}
                               -S^-_tv\left(x\right)\le  -S^-u\left(0\right) =
                              Su\left(0\right), \quad\forall x\in\rn.
                              \end{equation}
One the other hand, by Lemma \ref{LEM3.4}
for any $0<r<\frac{\dz_0}{R_Kt}$, there exists
      $z_r\in \overline {B\left(0,R_Ktr\right)}$ such that
     $$ \frac{u\left(z_r\right) }{tr}-tr L\big(\frac{z_r }{tr }\big)
     =S^+_{tr}u\left(0\right).$$

     For any $\ez>0$, there exists $r_{\ez,t}>0$ such that
     for any $r\in\left(0,r_{\ez,t}\right)$,
     $$v\left(y\right)\ge u_r\left(y\right)-\ez,\ \forall y\in \overline {B\left(0,R_Kt\right)}.$$
 Since $\frac{z_r}{r}\in \overline{B\left(0,R_Kt \right)}$, we obtain that
 for any $r\in\left(0,r_{\ez,t}\right)$,
$$S^+_tv\left(0\right)\ge \frac1t\Big[v\big(\frac{z_r}{r}\big)
-tL\big(\frac {z_r} {tr}\big)\Big]
\ge\frac1{tr}\Big[ u(z_r)  -tr L\big(\frac{   y }{t  }\big)\Big]
-\frac{\ez}t=  S^+_{tr}u\left(0\right)-\frac{\ez}t\ge Su\left(0\right)-\frac{\ez}t.$$
Sending $\ez$ to $0$, this implies that
 $$S^+_tv\left(0\right) \ge S u\left(0\right) .$$
 Similarly,  we have
$$ -S^-_tv\left(0\right) \ge  S u\left(0\right).  $$
Combining these two inequalities and \eqref{E26} and\eqref{E27},
we see that the assumptions of Theorem \ref{LEMC4} are satisfied,
with $k=Su\left(0\right)$. Hence  the conclusion ({\bf A})$\Rightarrow$
({\bf B})
follows from Theorem \ref{LEMC4}.
\end{proof}

We end this section with the following interesting Corollary. A function $v\in C^0(\rn)$ has the linear approximation property at the infinity if
for any sequence $\{R_j\}$ which converges to $\fz$, we can find a subsequence $\{ R_{j_k}\}_{k\in\nn}$
and a vector $e\in \rn$
such that
\begin{equation}\label{LAPfz}\lim_{k\to\fz}
\sup_{y\in B\left(0,1\right)}\big|\frac{u ( R_{j_k}y ) }{R_{j_k}} -  e\cdot y \big|=0
\end{equation}
and $H(e)=\|H(Dv)\|_{L^\fz(\rn)}$.

 \begin{cor}\label{lla-cor}  Let $n\ge 2$ and  assume that
$H\in C^0(\rn)$ is  convex  and coercive, and satisfies ({\bf A}) of Theorem 1.1. 
  Then the following hold: 
 \begin{enumerate}
\item[({\bf B}-1) ]   if $u\in AM_H(\rn)$ and   $Su(x_0)=\|H(Du)\|_{L^\fz(\rn)}<\fz$ for some $x_0\in\rn$, then $u$ is a linear function, with $H(Du(x_0))=\|H(Du)\|_{L^\fz(\rn)}$.

\item[({\bf B}-2) ]   if $u\in AM_H(\rn)$  has a linear growth at the infinity,
then $u$ has the linear approximation property at the infinity.
        \end{enumerate}
          \end{cor}

          \begin{proof}

          \medskip
          \noindent Assume that $u\in AM_H(\rn)$ and
          $Su(x_0)=\|H(Du)\|_{L^\infty(\rn)} <\fz$.
          For simplicity, assume that $x_0=0$ and $u(x_0)=0$.
          Then   $Su( 0) =\|H(Du)\|_{L^\fz(\rn)}:=k<\fz$. Since $u\in AM_H(\rn)$,
          it then follows from Lemmas \ref{LEM3.3} and \ref{LEM7.13} that
          the assumptions \eqref{F1} of Theorem \ref{LEMC4} are fulfilled.
          Hence we conclude that there exists $p_0\in\rn $, with $H(p_0)=k$,
          such that $u(x)=p_0\cdot x, \forall x\in\rn$. This proves ({\bf B}-1).

          To see ({\bf B}-2), assume that $u\in AM_H(\rn)$ has a linear growth at the infinity.
          Then $\|H(Du)\|_{L^\fz(\rn)}:=k<\fz$.
          For any sequence $R_{j}\to\fz$, there is
          function $v\in AM_H(\rn)$ such that after passing to a subsequence,          $u_{R_{j} } \to v $ locally uniformly in $\rn$. Similar to the proof of Theorem \ref{lla} (i)$\Rightarrow$(ii), we have
$Sv(0)=\|H(Dv)\|_{L^\fz(\rn)}\le k$.  By ({\bf B}-1), we know that $v$ is a linear
function.  \end{proof}

By the examples given in Section 4, we see that condition {\bf A} is also optimal (necessary in some sense) to get the properties
({\bf B}-1) and ({\bf B}-2).

\section{Proofs of Theorem \ref{xprop} and Theorem \ref{lla0}: Part I}

Now we will start to apply the linear approximation Theorem \ref{lla}
to deduce the $C^1$-regularity of absolute minimizers in dimension two.
This section is devoted to the proof of following result, which plays a crucial
role in the proof of Theorem \ref{lla0}. Here we follow
the argument by \cite{wy} by making all necessary technical
modifications.
\begin{thm}\label{xprop} For $n=2$,
assume $H$ satisfies both \eqref{assh0} and ({\bf A}) of Theorem \ref{lla}.
For each  $\ez >0$ and each vector  $e_8\in H^{-1}([1,2])  $
there exist  $ \dz_\ast(H,\ez,e_8 ) >0$  such that for any $0<\dz<\dz_\ast$,
if $u \in AM_H(B(0,8))$, then we have
 \begin{equation}\label{xlin3x}
 \max_{e\in \mathscr D u(0)}| e_8 -e |\le \ez\quad
 \mbox{whenever
  $e_8 \in \mathscr D(u )(0;8;\frac{\dz}8)$.}
\end{equation}
\end{thm}


\begin{proof}[Proof of Theorem \ref{xprop}]
For any $\dz>0$, let $ u \in AM_H(B(0,8))$,
$e_8\in\mathscr D(u)(0; 8;\frac{\dz}8)\cap H^{-1}([1,2])$,
and $e_{0,8}\in\mathscr D(u)(0)$.
We divide the proof into two cases:

\medskip
\noindent  {\it Case 1.  $u $ is nonlinear in any neighborhood of $0$.}
By $e_8\in\mathscr D(u)(0; 8;\frac{\dz}8)$ and  Lemma \ref{LEMest},
we have that
$$Su(x)\le H(e_8)+C_0\dz,\quad\forall x\in B(0,6),$$
and hence
\begin{equation}\label{w1}
\mbox{$H(e_{0,8})=Su(0)\le H(e_8)+C_0\dz\le 4$.    }
\end{equation}
provided $\dz<\frac{1}{C_0}$.

Set $$R:=1+ \max\big\{|p|: H(p)\le 4\big\}.$$
Let $\psi_{2R} $ be the function given by Proposition \ref{e-e6}.
Without loss of generality, assume that $\psi_{2R}(\ez)\le  \frac{\ez}4$.
Note that $e_{0,8},e_8\in \overline {B(0,R)}$.
If  $|e_{0,8}-e_8|\le \frac12\psi_{2R}(\frac{\ez}2) $,
then  we have $|e_{0,8}-e_8|\le \ez $ as desired.
Hence we assume that
$$|e_{0,8}-e_8|\ge \frac12\psi_{2R}(\frac{\ez}2).$$

Note that $e_{0,8}\in \mathscr D(u)(0)$ implies there exists $0<r<1/8$
such that $e_{0,8}\in \mathscr D(u)(0; 8r;\frac{\dz}8)$.
From the assumption that $u $ is nonlinear in $B(0,r )$,
there are a line segment $[z_1,z_2]\subset B(0,r)$
and a linear function $l(x)=a_0\cdot x+b_0, x\in [z_1,z_2]$,
with $a_0=[u(z_2)-u(z_1)]\frac{z_2-z_1}{|z_2-z_1|^2}$, and $z_3\in (z_1,z_2)$
such that either
 \begin{equation}\label{xcase1}
u\ge l\quad {\rm on}\ [z_1,z_2],\quad u(z_1)>l(z_1),\quad u(z_3)=l(z_3),\quad
u(z_2)>l(z_2);
\end{equation}
or
 \begin{equation}\label{xcase2}
u\le l\quad {\rm on}\ [z_1,z_2],\quad u(z_1)<l(z_1),\quad u(z_3)=l(z_3),\quad
u(z_2)<l(z_2).
\end{equation}
Since the case \eqref{xcase2} can be done similarly,
for simplicity we only consider \eqref{xcase1}.

Applying the linear approximation property \eqref{LAP} and \eqref{slope},
Lemma \ref{com} on comparison with linear functions,
and Corollary \ref{com-linear}, we can deduce the following result by
adapting Savin's topological argument in \cite{s05}, whose proof will
be given in Section 6.1 below.

\begin{lem}\label{xcon-com}
  There exists $e\in\mathscr D(u)(z_3)$
  such that $z_1$ and $z_2$ belong two distinct connected components
of the set $\{y\in \rr^2: u(y)>u(z_3)+e \cdot (y-z_3)\}\cap B(0,6)$.
\end{lem}

Since $e_{0,8}\in\mathscr D(u )(0;8r;\frac{\dz}8)$ and
  $B(z_3,2r )\subset B(0,6r)$, it follows from Lemma \ref{LEMest}
  and Lemma \ref{xcon-com} that
\begin{equation}\label{w2}H(e )=Su(z_3)\le H(e_{0,8})+C_0\dz.\end{equation}

The next Lemma gives a lower bound of $H(e)$
for $e$ given by Lemma \ref{xcon-com}, whose proof
will be given in Section 6.1.

 \begin{lem}\label{xlength}
 For every $\ez\in(0,1)$,
  there exists   $0< \dz_0(H,e_8,\ez )<\frac1{8C_0}\psi_{2R}\left(\frac\ez2\right) $ such that   for any $\dz\in(0,\dz_0)$,  we have
\begin{equation}\label{xeq4.3}
  H(e_8)-H(e )\le  \frac14 \psi_{2R}\left(\frac\ez2\right).
\end{equation}
\end{lem}

Hence if $\delta_0<\frac{1}{6C_0}\psi_{2R}(\frac{\ez}2)$ and $0<\dz<\dz_0$,
then by \eqref{w1}, \eqref{w2}, and Lemma \ref{xlength} that
$$\begin{cases}
|H(e_8)-H(e )| \le \frac14\psi_{2R}\left(\frac\ez2\right)+ 2C_0\dz<\ez,\\
|H(e_{0,8})-H(e )|\le \frac14\psi_{2R}\left(\frac\ez2\right)+2C_0\dz<\ez.
\end{cases}
$$
In particular, we have that $\frac12\le H(e_{0,8})\le 4$.

Next we need the following angle estimates; whose proofs
will also be given in Section 6.1.

\begin{lem}\label{xangle}   For every $\ez>0$ ,
 there exists   $0< \dz_1=\dz_1(H,e_8,\ez )<\dz_0 $ such that   for any $\dz\in(0,\dz_1)$, there exist $e'\in \overline {B\left(e,\psi_{2R}\left(\frac\ez2\right)\right)}$, with $H(e' )=H(e)$ and $q\in \partial H(e')$, such that
     \begin{equation}\label{xeq4.28}
\left|\measuredangle (q,e_8-e )-\frac{\pi}{2}\right|\le \psi_{2R}\left(\frac\ez2\right).
\end{equation}
\end{lem}

\begin{lem}\label{xangler}
 For every $\ez>0$,
  there exists   $ 0<\dz_2=\dz_2(H,e_8,\ez )<\dz_1 $ such that   for any $\dz\in(0,\dz_2)$, there exist $e'\in \overline {B\left(e,\psi_{2R}\left(\frac\ez2\right)\right)}$, with $H(e')=H(e)$ and  $q\in\partial H(e')$, such that
 \begin{equation}\label{xeq4.28r}
\left|\measuredangle (q,e_{0,8}-e)-\frac{\pi}{2}\right|\le \psi_{2R}\left(\frac\ez2\right).
\end{equation}
\end{lem}

Now we can prove \eqref{xlin3x}. Indeed,
 for every $\ez>0$ and $0<\dz<\min\{\dz_1, \dz_2\}$,
 Lemma \ref{xlength} and Lemma \ref{xangle} imply that
 the assumptions of Proposition \ref{e-e6} hold, with $p=e$ and $v=e_8-e$.
Hence
  $$| e_8-e| \le \frac{\ez}2. $$
Moreover, $|H(e_{0,8})-H(e )|<\psi_{2R}(\frac{\ez}2)$ and
  Lemma \ref{xangler} imply that
  the assumptions of Proposition \ref{e-e6} hold with $p=e$ and $v=e_{0,8}-e$. Hence
  $$| e_{0,8}-e | \le \frac{\ez}2.$$
Thus we have $| e_{0,8}-e_8| \le \ez$.

  \medskip
\noindent {\it Case 2. $u$ is linear in some neighborhood of $0$.}
In this case, we have that $e_{0,8}=Du(0)$ and $u(x)= u(0)+ e_{0,8}\cdot x$ for $x$ near $0$. Let $r_0\in(0,8]$ be the largest $r\in(0,8)$ so that
 $u(x)=u(0)+ e_{0,8}\cdot x$ for $x\in B(0,r )$. Hence
 $u(x)=u(0)+e_{0,8}\cdot x$ for $x\in B(0,r_0)$.

If $r_0 >\frac18$, then
by $e_8\in\mathscr D(0;8;\frac{\dz}8)$ we have
$$\sup_{x\in B(0,2)}|e_{0,8}\cdot x-e_8\cdot x|\le \dz. $$
Hence $ |e_{0,8}-e_8  |\le \ez$ whenever $0<\dz<\ez$.

If $r_0< \frac18$,
the definition of $r_0$ implies there exists $x_0\in\partial B(0,r_0)$
such $u(x)\ne u(0)+e_{0,8}\cdot x$ in any neighborhood of $x_0$.
We claim that $u$ is nonlinear in any neighborhood of $x_0$.
For, otherwise, there exist $s>0$ and $p\in\rr^2$
such that $u(x)=u(x_0)+ p\cdot (x-x_0)$ for $x\in B(x_0,s)$.
Then we can see that
$$(e_{0,8}-p)\cdot (x-x_0) =0, \ \forall x\in B(0,r_0)\cap B(x_0,s).$$
 If $e-p\perp x_0$, then we must have that $e_{0,8}=p$, which is impossible.
 If $e-p$ is not perpendicular to $x_0$, then we can find $w\in B(0,r_0)\cap B(x_0,s)$ such that either $\frac{w-x_0}{|w-x_0|}$ or
 $\frac{x_0-w}{|w-x_0|}$
 equals to $\frac{e_{0,8}-p}{|e_{0,8}-p|}$.
Hence we also have $e_{0,8}=p$, which is impossible..

 Next we claim that $\mathscr D(u)(x_0)=\{e_{0,8}\}$, that is, $u$ is differentiable at $x_0$ and $e_{0,8} = Du(x_0)$.
 In fact, by ({\bf B}) of Theorem \ref{lla},  for any
  sequence $r_j\to 0$, there exists a vector $e_{x_0,{\{r_j\}} }$ such that $H( e_{x_0,{\{r_j\}}})=Su(x_0)$ and
$$\lim_{ j\to \fz}\sup_{y\in B(x_0,r_j)}\frac{ |u(y)-u(x_0)-e_{x_0, {\{r_j\}}}
\cdot  (y-x_0)  | }{r_j}=0.$$
Since $u(y)=u(0)+e_{0,8}\cdot y$ for $y\in \overline {B(0,r_0)}$, we have
 $$\lim_{ j\to \fz}\sup_{y\in B(x_0, r_j)\cap B(0,r_0)}\frac{ |(e_{0,8} -e_{x_0,
  {\{r_j\}} })\cdot (y-x_0) | }{r_j}=0.$$
 If $(e_{0,8} -e_{x_0, {\{r_j\}} })\cdot x_0=0$, then we must have that
  $|e_{0,8} -e_{x_0, \{r_j\} }|=0$. \\
 If $(e_{0,8} -e_{x_0, {\{r_j\}} })\cdot x_0\ne 0$,
  then for sufficiently large $j$,  we can find vectors
  $x_j \in \partial B(x_0, r_j)\cap B(0,r_0)$ such that either
  $ \frac{x_j-x_0}{|x_j -x_0 |} $ or $-  \frac{x_j-x_0} {|x_j-x_0 |}$ equals to
  $ \frac{e_{0,8} -e_{x_0, {\{r_j\}}}}{|e_{0,8} -e_{x_0, {\{r_j\}} }|}$.
 Hence $|e_{0,8} -e_{x_0, \{r_j\}}|=0$, that is, $e =e_{x_0, {\{r_j\}} }$.
  We then conclude that $\mathscr D(u)(x_0)=\{e_{0,8}\}$.

Finally, observe that $e_8\in \mathscr D(u)(0;8; \frac{\dz}8)$ implies that
$e_8\in \mathscr D(u)(x_0;7; \frac{2\dz}{7})$.
Define $v(x)=\frac87 [u(\frac 78 x+x_0)-u(x_0)]$ for $x\in B(0,8)$.
Then $v\in AM_H(B(0,8))$, $e_8\in \mathscr D(v)(0;8; \frac{3\dz}{8})$, and $Dv(0)=e_{0,8}$.
If $\dz<\frac13\min\{\dz_1, \dz_2\}$,
then we can argue as in Case 1 to conclude
that $|e_8-e_{0,8}|\le \ez$ as desired.
The proof of Theorem \ref{xprop} is complete, given Lemmas
6.2, 6.3, 6.4, and 6.5.
\end{proof}

\subsection{Proofs of Lemmas \ref{xcon-com}, \ref{xlength}, \ref{xangle}, and \ref{xangler}}

  \begin{proof}
[Proof of Lemma  \ref{xcon-com}.] By Theorem \ref{lla} ({\bf B}), there exists
 $e\in\mathscr D(z_3)$, with $H(e)=Su(z_3)$, such that for some
 sequence $s_k\rightarrow 0$,
\begin{equation}\label{xeq4.1}
\lim_{ k\to \fz}\sup_{B(z_3,s_k)}\frac{ |u(y)-u(z_3)-e\cdot (y-z_3) | }{s_k}=0.
\end{equation}
For any $z_k\in  [z_1 ,z_2 ]\cap \partial B(z_3,s_k)$, by
\eqref{xcase1} and \eqref{xeq4.1}, we have
\begin{align*}
 \left(a_0-e\right) \cdot \left(\frac{z_k-z_3}{s_k}\right)=
\frac{l(z_k)-l(z_3)-e\cdot (z_k-z_3)}{s_k}\le
\frac{u(z_k)-u(z_3)-e\cdot (z_k-z_3)}{s_k}\to 0,
\end{align*}
as $k\to \fz$. Therefore
 $$\left(a_0-e\right) \cdot \left( z -z_3 \right)=0 \quad\forall z\in [z_1,z_2].$$
Applying \eqref{xcase1} again, we then obtain
$$u(z_i)-u(z_3)>l(z_i)-l(z_3)=a_0\cdot (z_i-z_3)=e\cdot
(z_i-z_3),\quad i=1,2,$$
so that  $z_1, z_2\in \big\{y \in \rr^2 : u(y)>u(z_3)+e\cdot(y-z_3 )\big\}.$

Now suppose that $z_1$, $z_2$ were in the same connected component of
$$\big\{y \in \rr^2 : u(y)>u(z_3)+e\cdot(y-z_3 )\big\}\cap B(0,6).$$
Then there would exist a simple curve $\gamma_0\subset \{y \in \rr^2 : u(y)>u(z_3)+e\cdot(y-z_3 )\}
 \cap B(0,6)$ joining $z_1$ to $z_2$. Let $\gz=\gamma_0\cup
 [z_1,z_2]$ be the simple closed curve and $U\subset B(0,6)$ be the open set bounded by $\gz$ so that $\gz=\partial U$.
Without loss of generality, we may assume that
there exists a small $\beta >0$ such that either
 $$B^+(z_3,\beta):=B(z_3,\beta)\cap
 \{y\in\rr^2:0<\measuredangle (y-z_3,z_2-z_1)<\pi\}\subset U.$$
 Let $ \nu\in \rr^2$, with $|\nu|=1$, be  such  that $\nu\cdot(z_2-z_1)=0$ and $z_3+\frac12\bz\nu \in U$. Since there exists an $\dz_0$ such that
 $$u(y)-u(z_3)-e\cdot (y-z_3)\ge \dz_{0}, \ \forall y\in \gamma_0,$$
we can find  a small
$\ez_0>0$  such that
 $$u(y)\ge u(z_3)+(e+\ez_0 v)\cdot (y-z_3),
\, \forall y\in \gz_0.$$
From $\nu\cdot(z_1-z_2)=0$, we also see that
$$u(y)\ge u(z_3)+(e+\ez_0 v)\cdot (y-z_3), \ \forall y\in [z_1,z_2].$$
Thus by Lemma \ref{com} and Corollary \ref{com-linear}
we conclude that
$$u(y)\ge u(z_3)+(e+\ez_0 v)\cdot (y-z_3),\quad y\in U.$$
This implies that
$$\lim_{k\to\fz}\max_{y\in B^+(z_3,\beta)\cap B(z_3,s_k)}
\frac{|u(y)-u(z_3)-e\cdot (y-z_3)|}{s_k }\ge
\frac{ u(z_3+s_k\nu)-u(z_3)-s_ke\cdot  \nu }{s_k }\ge \ez_0  >0,$$
which contradicts to \eqref{xeq4.1}. This completes
the proof of Lemma \ref{xcon-com}.
\end{proof}

In order to prove Lemma \ref{xlength}, set
$$\eta=\eta(\ez):=\frac1{4C_1}\psi_{2R}(\frac{\ez}2),$$
where
$$C_1:=1+\max\big\{|q|, q\in\partial H(p), H(p)\le 4\big\}.$$
Let $f=e_8-e$.
If $|f|\le \eta$, then
$$H(e)\ge  H(e_8)-q\cdot f\ge H(e_8)- C_1 \eta \ge H(e_8)
-\frac14\psi_{2R}(\frac{\ez}2),\quad\forall q\in\partial H(e_8),$$
 which implies \eqref{xeq4.3}.

Below we assume $|f|\ge \eta $.
Set
$$
\mathscr{S}\equiv\{y\in\rr^2: |f\cdot (y-z_3)|\le
2\dz\}, \mathscr{S}_-\equiv\{y\in\rr^2: f\cdot (y-z_3)<-2\dz\}, \ {\rm and}\ \mathscr{S}_+\equiv\{y\in\rr^2: f\cdot (y-z_3)> 2\dz\}.$$
The  width of $ \mathscr{S}$ is $\frac{2\dz}{|f|}\le\frac{2\dz}{\eta}$.
Moreover,  since $e_8\in\mathscr D(u)(0; 8;\frac{\dz}8)$,
we have that
 $$|u(y)-u(z_3)-e_8\cdot (y-z_3)|\le 2\dz,\quad \forall y\in
 B(0,6).$$
Since
$$ \mathscr{S}_- \cap B(0,6)
\subset \big\{y\in\rr^2: u(y)<u(z_3)+e\cdot (y-z_3)\big\} $$
and
$$\mathscr{S}_+ \cap B(0,6)
\subset \big\{y\in\rr^2: u(y)>u(z_3)+e\cdot (y-z_3)\big\},$$
it follows from Lemma \ref{xcon-com} that
there is a connected component $U$ of
$\{y\in\rr^2:u(y)>u(z_3)+e\cdot (y-z_3)\}\cap B(0,6)$, containing either $z_1$
or $z_2$, such that $U\subset\mathscr{S}$ and $U\cap B(0,1)\ne\emptyset$.
Moreover, $U\nsubset B(0,6)$. For, otherwise,   since
$$u(y)=u(z_3)+ e\cdot (y-z_3)\ {\rm{on}}\ \partial U,$$
it follows from Lemma \ref{com-linear}
that
$$u(y)=u(z_3)+ e\cdot (y-z_3)\ {\rm{in}}\ U,$$
which  contradicts to the definition of $U$.

Therefore, there exists a polygonal line $\Gamma\subset U$, which
starts inside $B(0,1)$ and ends outside $B(0,6)$.
There exists $z_4\in B(0,6)$,  with $|z_4-z_3|=3$ and $z_4-z_3\perp f$,
 such that (see Figure 3 below)

\begin{enumerate}
\item[(A1)]  $$\displaystyle\sup_{y\in B(z_4,2)}
\big|u(y)-u(z_3)-e_8\cdot (y-z_3) \big| \le 2\dz, $$
and $1\le H(e_8)\le 2$.

\item[(A2)]   $\big\{y\in\rr^2:u(y)>u(z_3)+e\cdot (y-z_3)\big\}\cap
B(0,6)$ has a connected component  $U\subset  \mathscr{S}$
that contains a polygonal line
$\Gamma$ connecting the two arcs $ \mathscr{S} \cap \partial B(z_4,2)$.
\end{enumerate}

\begin{center}
\begin{tikzpicture}

\fill (0,0) circle (1pt);
\draw (0,0) node [left] {$O$};

\draw (-4.5,0.7) -- (4,0.7) -- cycle;
\draw (-4.5,-0.7) -- (4,-0.7) -- cycle;

\draw   [color=green]  (4,0.32)--(3,0.25)--(2,0.3)--(1,0.36)--(0.5,0.2)--
(0.3,0.15)--(0.2 ,0.1)--(0.1,0.05)--(0.3, -0.1 )--(0.4, -0.3)
--(0.5,-0.250)--(1,-0.2)--(1.5,-0.19)--(2,-0.28)--(3,-0.2)--
(4,-0.17) ;

\fill [color=green,opacity=0.5] (4,0.32)--(3,0.25)--(2,0.3)--(1,0.36)--(0.5,0.2)--
(0.3,0.15)--(0.2 ,0.1)--(0.1,0.05)--(0.3, -0.1 )--(0.4, -0.3)
--(0.5,-0.250)--(1,-0.2)--(1.5,-0.19)--(2,-0.28)--(3,-0.2)--
(4,-0.17);



\draw  [ ]plot[smooth] coordinates
{(-4.5,-0.5) 
(-3,-0.46) (-2,-0.36) (-1,-0.53) (0,-0.60) (1,-0.64) (2,-0.5)
 (3,-0.66) (4,-0.6)}; 



\draw  [ ]plot[smooth] coordinates
{(-4.5,0.66) 
(-3,0.54) (-2,0.35) (-1,0.52) (0,0.60) (1,0.47) (2,0.5)
 (3,0.51) (4,0.56)};

\draw  (0.6,  -0.15)
 --(1,0.08)--(1.5,0.12)--(2,-0.16 )--(3, 0.1)--
(4, 0 ) ; 
\draw (3.4,0.1) node [below] {$\Gamma$};
 \draw (4,0.2) node [right] {$U$};


\draw (2,0.2) circle (1cm);
\fill (2,0.2) circle (1pt);
\draw (2,0.2) node [right] {$z_4$};

\fill (0.5,0.2) circle (1pt);
\draw (0.55,0.3) node [left] {$z_3$};

\fill (0.4,0.0) circle (1pt);
\draw (0.4,0.0) node [below] {$z_1$};

\fill (0.7,0.6) circle (1pt);
\draw (0.7,0.6) node [above] {$z_2$};

\draw  (0.4,0.0)--(0.7,0.6);

\draw  (0.5,0.2)--(2,0.2);

\draw (-4.2,-1.2) node [right]{$\mathscr{S}_-$};
\draw (-4.2,1.2) node [right]{$\mathscr{S}_+$};
\draw (-4.2,0) node [right]{$\mathscr{S}$};

\draw[->,>=latex, thick, color=blue](0,0) --(0,1);
\draw (0,1) node[right]{$f$};
\draw  [dotted, thick, color=gray] (-4.3,-0.7)--(-4.3,0.7)-- cycle;
\draw [->,>=latex](-4.3,-0.6) --(-4.3,-0.7);
\draw [->,>=latex](-4.3,0.6) --(-4.3,0.7);
\draw (-4.4,0) node[left]{$\frac{2\dz}{|f|}$};


\draw [dotted, thick, color=gray] (2,0.2) -- (1.5,1.12)-- cycle;
\draw (1.6,1.2) node[right] {$2$};
%

\draw (0,-1.2) node [below] {${\rm Figure \ 3}$};
\end{tikzpicture}
\end{center}


We now have
\begin{lem}\label{xlengthU} It holds that
 \begin{equation}\label{xeq4.6}
  Su(x)\le H(e)+4C_0\dz, \quad\forall x\in U\cap B(z_4,1).
  \end{equation}
\end{lem}

\begin{proof}
[Proof of Lemma \ref{xlengthU}.]

 %
%
%
%
%

 %
%
%

 For any $x_0\in B(z_4,1)\cap U$, we have $B(x_0,1)\subset B(z_4,2)$ and
 $u(z_3)+e\cdot (x_0-z_3)<u(x_0)$. Observe that
\begin{equation}\label{xeq4.7}
u(y)=u(z_3)+
e\cdot (y-z_3)\le u(x_0)+
C^H_{H(e)}(y-x_0),\quad \forall y\in \partial U\cap
B(x_0,1),
\end{equation}
and
\begin{equation}\label{xeq4.8}
u(y)\le u(z_3)+
e\cdot (y-z_3)+2\dz\le u(x_0)+
e\cdot (y-x_0)+4\dz,\quad \forall y\in U\cap
\partial B(x_0,1).
\end{equation}
By \eqref{xew5}, \eqref{xeq4.7}, and \eqref{xeq4.8}, there exists
$C_0>0$ such that
\begin{equation}\label{xeq4.9}
u(y)\le u(x_0)+ C^H_{H(e)+4C_0\dz}(y-x_0),
\ \forall y\in\partial (U\cap B(x_0,1)),
\end{equation}
and hence
\begin{equation}\label{xeq4.9'}
u(y)\le u(x_0)+C^H_{H(e)+4C_0\dz}(y-x_0),
\ \forall y\in U\cap B(x_0,1).
\end{equation}
This, combined with Lemma \ref{LEMest}, implies \eqref{xeq4.6}.
\end{proof}

We now return to the proof of Lemma \ref{xlength}.

\begin{proof}
[Completion of proof of Lemma \ref{xlength}.]  (Also see Figure 4 below for illustration).
We may also assume that $H(e)\le H(e_8)-\eta$.
Let $\dz<\frac{1}{100C_1}$ so that
 $$50| q |\dz\le 50C_1\dz\le \frac12 H(e_8), \ \forall q\in \partial H(e_8).$$
The convexity of $H$ implies
\begin{equation}\label{xeq4.4}
 q \cdot f=  q \cdot (e_8-e )\ge H(e_8)-H(e)\ge \eta,\quad \quad\forall  q \in \partial H(e_8).
\end{equation}

 \medskip
\noindent {\it Claim  I.
For any $ q \in \partial H(e_8)$, there exists
$$w_{ q }\in I:= \Big\{z_4+s\frac{ q }{| q |}: -
\frac 13\le s\le -\frac 14\Big\}\subset B(z_4,1)$$ such that
\begin{equation}\label{xeq4.10}
Su(w_{ q })\ge H(e_8)-50 | q |\dz.
\end{equation}}

\begin{proof}
[Proof of  Claim  I]
Let $ q \in \partial H(e_8)$
and $c:=\sup_{x\in I} Su(x)$.   To find such a   $w_q$,
it suffices to show that
$$H(e_8)-48 | q |\dz \le c,$$
which follows from
\begin{equation}\label{xeq4.14}
e_8\cdot  q -48| q |\dz \le C^H_c( q ).
\end{equation}
Indeed, since $q\in\partial H(e_8)$, we have that
$e_8\cdot  q=H(e_8)+L(q)$, and hence
$$ C^H_c( q )=\sup_{H(p)\le c} p\cdot q\le
\sup_{H(p)\le c} [H(p)+L(q)]=c+L(q)=c+e_8\cdot q-H(e_8).
 $$

To see \eqref{xeq4.14}, observe that, thanks to $I\subset B(z_4,2)$,
(A1) implies that
\begin{equation}\label{xeq4.11}
u\big(z_4-\frac{ q }{4| q |}\big)-
u\big(z_4-\frac{ q }{3| q |}\big)\ge
\frac1 {12} e_8\cdot \frac{ q }{| q |}-4\dz.
\end{equation}
On the other hand,
by the upper semicontinuity of $Su$, for any $\eta>0$
 there exists an open neighborhood  $V_\eta(I)$  such that
 $$\sup_{x\in V_\eta(I)} H(Du(x))=\sup_{x\in V_\eta(I)} Su(x)\le c+\eta.$$
Hence by Lemma \ref{LEM7.11}, we have
$$
u\big(z_4-\frac{ q }{4| q |}\big)-
u\big(z_4-\frac{ q }{3| q |}\big)\le C^H_{c+\eta}(\frac{ q }{12| q |}).
$$
This and \eqref{xeq4.11}, after taking $\eta\to0$, yield \eqref{xeq4.14}.\end{proof}

\begin{center}
\begin{tikzpicture}
\fill (0,0) circle (1pt);
\draw (0,0.1) node [left] {$z_4$};
\draw (0,0) circle (4 cm);
\draw  [dotted, thick, color=gray] (0,0) -- (3,2.64)-- cycle;
\draw (4,2.8) node [left] {$r=2$};
\draw (-4.5,1.7) node [right] {$ \mathscr{S}_+$};
\draw (-4.5,0.8) node [right] {$ \mathscr{S} $};
\draw (-4.5,-1.7) node [right] {$ \mathscr{S}_-$};

\draw (-5,1) -- (5,1) -- cycle;
\draw (-5,-1) -- (5,-1) -- cycle;
\draw[->,>=latex,thick, color=blue](0,0) --(0,2);
\draw (0,2) node[right]{$f=e_8-e$};
\draw  [dotted, thick, color=gray] (-4.7,-1)--(-4.7,1)-- cycle;
\draw [->,>=latex, thick](-4.7,-0.9) --(-4.7,-1);
\draw [->,>=latex,thick](-4.7,0.9) --(-4.7,1);
\draw (-5,0) node[left]{$\frac{2\dz}{|f|}$};

\draw [thick]plot[smooth] coordinates
{(-4.5,0.4) (-4.2,0.34) (-4,0.4)(-3.8,0.3)(-3.5,0.4)(-3,0.35)(-2.5,0.3)
(-2.3,0.4)(-2.2,0.43)(-2,0.44)(-1.5,0.35)(-1,0.3)(-0.5,0.4)
(0,0.43)(0.5,0.3)(1,0.35)(1.5,0.4)(2,0.38)(2.5,0.43)(3,0.44)(3.5,0.39)
(4,0.35)(4.5,0.44)};
\draw [thick]plot[smooth] coordinates
{(-4.5,-0.41) (-4.2,-0.30) (-4,-0.39)(-3.8,-0.4)(-3.5,-0.35)(-3,-0.4)(-2.5,-0.45)
(-2.3,-0.5)(-2.2,-0.51)(-2,-0.3)(-1.5,-0.35)(-1,-0.37)(-0.5,-0.41)
(0,-0.37)(0.5,-0.35)(1,-0.38)(1.5,-0.45)(2,-0.40)(2.5,-0.41)(3,-0.37)(3.5,-0.42)
(4,-0.45)(4.5,-0.36)};


\draw[] (-4.3,0.2)--(-4,0)--(-3.5,0.2)--(-3,0.1)--(-2.5,-0.2)--(-2,0.1)--(-1.5,-0.21)
--(-1,0.1)--(-0.5,0)--(0,-0.2)--(0.5,0.12)--(1,0.2)--(1.5,-0.1)
--(2,0.1)--(2.5,-0.05)--(3,-0.25)--(3.5,-0.24)--(3.99,-0.3)--(4.3,0.1);
\draw (-4,-0.1) node[right]{$\Gamma$};
\draw [->,>=latex,thick](0,0) -- (-1.5,-2);
\draw (-1.5,-2) node[right]{$y_0=w_{q_0}$};

\draw[thick, color=red] (-1.5,-2)--(-1.45,-1.5)--(-1.24,-1.2)--(-1.3,-1)--(-1.1,-0.8)
--(-1.24,-0.5)--(-1.0,0)--(-1.25,0.3)--(-0.9,0.5)--(-1.1,0.6)--(-1.12,0.8)
--(-0.9,1.3)--(-1,1.7)--(-0.85,2.3)--(-0.9,2.7)--(-0.7,3.3)--(-0.9,3.4)
--(-0.6,3.7);
\fill (-0.6,3.7) circle (1pt);
\draw (-0.6,3.7) node[right]{$y_m$};
\draw (3,0.2) node[right]{$U$};
\draw (0,-4.2) node [below] {${\rm Figure \ 4}$};
\end{tikzpicture}
\end{center}

Next  let  $ q^0\in \partial H(e_8)$ be such that
\begin{equation}
\label{EQ4.x1}f\cdot \frac{q^0}{|q^0|}=\min\Big\{f\cdot \frac{q }{|q |}: q\in \partial H(e_8)\Big\}.
\end{equation}
Let
$y_0:=w_{ q^0}=z_4+s_0\frac{q^0}{|q^0|}$, for some
 $s_0\in[-\frac13,-\frac14]$,  be given by {\it Claim I}.
 Then it follows from \eqref{xeq4.4} and $z_4-z_3\perp f$
 that
  $$
 f\cdot (y_0-z_3)=  f\cdot (y_0-z_4)= s_0f\cdot \frac{q^0}{|q^0|}
 <-\frac14\frac\eta {|q^0|}<-2\delta,$$
 provided $\dz<\frac{\eta}{8|q^0|}$.
 Hence $y_0\in\mathscr S_-$.

For any $0<\dz<\frac{\eta}{8|q^0|}$,  let
 $$t=t(\dz):={\rm dist}(\Gamma, \partial U\cap B(z_4,2))\le \frac{2\dz}{|f|}.$$
Applying Lemma \ref{LEMconeincreasing}, we obtain a discrete gradient flow $\{y_i \}_{i=0}^{m}$ for some $m=m(\dz)$, satisfying
\begin{equation}\label{xeq4.18}
y_i=y_i(\dz)\in B(z_4,2) ,\ |y_i-y_{i-1}|=t,\ u(y_i)=u(y_{i-1})+
C^H_{S^+_t( u)(y_{i-1})}(y_i-y_{i-1}),\quad 1\le i\le m,
\end{equation}
but $y_{m+1}\notin B(z_4,2)$, which implies that
\begin{equation}\label{xeq4.19}
{\rm dist}(y_m,\partial B(z_4,2))\le t.
\end{equation}
To see the existence of such a $m\ge 1$, we argue by contradiction.
Assume \eqref{xeq4.18} holds for all $i\ge1$. Then, by Claim I,
\begin{equation}\label{xeq4.20}
S ( u)(y_i)\ge S ( u)(y_0)\ge H(e_8)-50|  q^0|\dz,\ \forall i.
\end{equation}
Thus, for any $i\ge1$
\begin{align}\label{xeq4.20x}
u(y_i)-u(y_0)&=\sum^i_{j=1}\left(u(y_j)-u(y_{j-1})\right)
\ge \sum^i _{j=1} C^H_{S ( u)(y_{j-1})}(y_j-y_{j-1})\\
&\ge \sum^j_{i=1}C^H_{S ( u)(y_0)}(y_j-y_{j-1})
\ge \sum^i_{j=1}C^H_{H(e_8)-50| q^0|\dz}(y_j-y_{j-1}).\nonumber
\end{align}
Note that by $50| q^0|\dz\le \frac12 H(e_8)$,
 there exists $C_2(H)>0$    such that
$$C^H_{H(e_8)-50| q^0|\dz}(x)\ge C^H_{\frac12   }(x)\ge
C_2|x|, \quad \forall |x|=1,$$
hence
\begin{align*}
u(y_i)-u(y_0)
&\ge \sum^i_{j=1}C^H_{H(e_8)-50| q^0|\dz}(y_j-y_{j-1})
\ge Cit.\nonumber
\end{align*}
This, combined with $|u(y_i)-u(y_0)|\le \|Du\|_{L^\fz(B(0,6))}
|y_i-y_0|$, implies that
 $|y_i-y_0|\ge Cit $ holds for all $i\ge 1$, which is impossible.

 \bigskip

\noindent {\it Claim  II. There exists  $0<\dz(H,e_8,\ez)<\eta/8$ such that for any $0<\dz<\dz(H,e_8,
\ez)$, we can find $1\le j_\dz\le m-1$ such that  $y_{j_\dz}\in B(z_4,1)\cap U$
and $y_m\in \mathscr S_+$.}

\begin{proof}
[Proof of Claim II]
We first show that if $\dz>0$ is sufficiently small, then $y_m\in \mathscr S_+$.
To see this, observe that
\begin{equation}\label{xeq4.23}
C^H_{H(e_8)-50| q^0|\dz}(y_m-y_0)\le e_8\cdot (y_m-y_0)+4\dz.
\end{equation}
In fact, applying \eqref{xeq4.20x} with $i=m$ and
the triangle inequality for $C^H_k$, we have that
\begin{align}\label{xeq4.21}
u(y_m)-u(y_0)&\ge C^H_{H(e_8)-50| q^0|\dz}(y_m-y_0).
\end{align}
While, by (A1) and $y_m,y_0\in B_2(z_4)$ we have
\begin{equation}\label{xeq4.22}
u(y_m)-u(y_0)\le e_8\cdot (y_m-y_0)+4\dz.
\end{equation}
\eqref{xeq4.23} follows from \eqref {xeq4.22} and \eqref{xeq4.21}.

From $\dist(y_m,\partial B(z_4,2))\le t$, there exists
$y_*\in \partial B(z_4,2)$ such that $t\ge |y_m-y_*|$.
For $0<\dz<\frac{\eta}8$, by $|f|\ge \frac{\eta}2$
and  $t<\frac{2\dz}{|f|}\le  \frac23$, it holds that
$$ |y_m-z_4|\ge |y_*-z_4|-|y_*-y_m|\ge 2-t\ge   \frac 43. $$
This, combined with $y_0=w_{q^0}\in \{z_4+t\frac{ q^0}{| q^0|}: t\in[-\frac 13,-\frac 14]\}$, implies that
$$|y_m-y_0|\ge |y_m-z_4|-|z_4-y_0|\ge \frac 43-\frac 13=1.$$

Denote $e_m=\frac{y_m-y_0}{|y_m-y_0|}$. Then by
\eqref{xeq4.23}   we have
\begin{equation}\label{xeq4.24}
C^H_{H(e_8)-50| q |\dz}(e_m)\le e_8\cdot e_m+4\dz.
\end{equation}
{Let $\tau(  e_8,4,\eta/16C_1|q^0|)>0$ be given by Lemma \ref{xxlem} and
$0<\delta<\tau( e_8,4,\eta/3|q^0|)$. Then Lemma \ref{xxlem} implies
that there exists $\hat q\in \partial H(e_8)$ such that
$$|e_m-\frac{\hat q} {|\hat q|}|\le \frac{\eta }{16C_1|q^0|}.$$}

 Next we show that if  $\dz>0$  is sufficiently small, then
$y_m\in \mathscr S_+=\{y\in\rr^2:f\cdot (y-z_3)\ge 2\dz\}$.
Indeed,  since  $y_0 =z_4+s_0\frac{q^0}{|q^0|}$ for some
 $s_0\in[-\frac13,-\frac14]$,  we have
$$
f\cdot (y_m-z_4)= f\cdot (y_m-y_0)+f\cdot (y_0-z_4)= |y_m-y_0|
 f\cdot \frac{\hat q}{|\hat q|}  + |y_m-y_0| f\cdot (e_m-\frac{\hat q}{|\hat q|})+s_{ 0} f\cdot \frac{q^0}{|q^0|}.$$
{Since $1\le |y_m-y_0|\le 4$,
$f\cdot \frac{\hat q}{|\hat q|}\ge f\cdot \frac{q^0}{|q^0|}\ge \frac{\eta}{|q^0|}$ and $|f|\le 2C_1$, we get
$$
f\cdot (y_m-z_4)\ge   \frac23
 f\cdot \frac{q^0}{|q^0|} -4|f||e_m-\frac{\hat q}{|\hat q|} |\ge \frac{2\eta}{3|q^0|}-8C_1 \frac{\eta }{16C_1|q^0|} \ge \frac{ \eta}{6|q^0|}\ge 2\dz,$$
if we choose $\dz<\frac{ \eta}{12|q^0|}$.
Since $f\cdot(y_m-z_3)= f\cdot(y_m-z_4)$, we conclude
that $y_m \in \mathscr S_+$.}

For a sufficiently small $\dz>0$,
it follows from $y_0\in \mathscr S_-$, $y_m\in \mathscr S_+$,
and the choice of the step size $t$ that
there exists $1\le j_\dz\le m-1$ such that $y_{j_\dz}\in U\cap B(z_4,2)$.
It remains to show that $|y_{j_\dz}-z_4|\le 1$.
For, otherwise, we have $|y_{j_\dz}-z_4|> 1$ so that
$$|y_ {j_\dz}-y_0|\ge |y_{j_\dz}-z_4|-|y_0-z_4|\ge 1-\frac{1}{3}=\frac 23.$$
Then, by an argument similar to the above, we
can show that  $ f\cdot(y_{j_\dz}-z_4)>2\dz$, that is,
$y_{j_\dz}\in \mathscr S_+$, which is a contradiction.
This proves  {\it Claim  II}.
\end{proof}

It follows from \eqref{xeq4.6}, \eqref{xeq4.10},
\eqref{xeq4.20}, and {\it Claim  II} that
$$H(e_8)-50| q |\dz\le S  u(y_{j_\dz})\le H(e)+C\dz,$$
this implies that \eqref{xeq4.3}, where $\dz=\dz(\eta,H)$ is chosen to be sufficiently small. The proof of Lemma \ref{xlength} is now complete.
\end{proof}

Next we will give a proof of Lemma \ref{xangle}.

\begin{proof}
[Proof of Lemma \ref{xangle}] (Also see Figure 5 below for illustration).
If there exist $e'\in B(e,\eta)$, with $H(e)=H(e')$,
and $q_1,q_2 \in \partial H(e')$ such that
 $$\measuredangle (q_1 ,f)\le \frac{\pi}2\le \measuredangle (q_2 ,f),$$
then we can find $\lz\in[0,1]$ such that
$\measuredangle(\lz q_1+(1-\lz)q_2,f)=\frac{\pi}2$,
which satisfies \eqref{xeq4.28}.

Hence, without loss of generality we may assume
\begin{equation}\label{x1}
\measuredangle (q ,f)\in[0,\frac{\pi}{2}),
\forall q \in \partial H(e') \ {\rm{for\ all}}\ e'\in B(e,\eta)\ {\rm{with}}\ H(e)=H(e').
\end{equation}
Indeed, Lemma \ref{xangle} can be similarly proved for the case that $\measuredangle (q ,f)\in(\frac{\pi}2, {\pi} ]$
for all $q \in \partial H(e')$ and all $e'\in B(e,\eta)$ with $H(e)=H(e')$.

 Note that  there exists $e'\in \overline {B(e,\eta)}$, with $H(e)=H(e')$,
 and $ q_{e'}\in \partial H(e')$ such that
 $$\alpha:=\measuredangle (q_{e'} ,f)=\max_{e'\in {\overline B(e,\eta)}}\max_{q\in\partial H(e')}\measuredangle (q ,f)\le \frac{\pi}{2} .$$
Suppose that the conclusion of Lemma \ref{xangle}
were false. Then we would have $\alpha< \frac{\pi}{2} -  \eta$.
Let $x_\dz=z_4-\frac{2q_{e'}}{q_{e'}\cdot f}\dz$ be the intersection point
between $L:=\{z_4+sq_{e'} :t\in\rr\}$ and $\{y\in\rr^2:(y-z_4)\cdot f=-2\dz\}$.
Observe  that
$$|x_\dz-z_4|=\frac{2\dz}{|f|\cos \measuredangle (q_{e'} ,f)}\le
\frac{2\dz}{\eta \sin\eta}\le1, $$
provided $\dz >0$ is chosen to be sufficiently small.
This implies $B(x_\dz,1)\subset B(z_4,2)$. By (A1), we have
\begin{equation}\label{xeq4.29}
u(y)-u(z_3)-e\cdot (y-z_3)\le u(y)
-u(z_3)-e_8\cdot (y-z_3)+f\cdot (y-z_3)\le 4\dz,
\quad \forall y\in U\cap B(x_\dz,1).
\end{equation}
On the other hand, we have
\begin{eqnarray}\label{xeq4.30}
u(y)&=&u(z_3)+e\cdot (y-z_3)\nonumber\\
&=&u(z_3)+e\cdot (x_\dz-z_3)+e\cdot (y-x_\dz)\nonumber\\
&\le& u(z_3)+e\cdot (x_\dz-z_3)+C^H_{H(e)}(y-x_\dz),
\quad \forall y\in \partial U\cap B(x_\dz,1).
\end{eqnarray}

\begin{center}
\begin{tikzpicture}
\fill (0,0) circle (1pt);
\draw (0,0.1) node [left] {$z_4$};

\draw (-4.5,1.6) node [right] {$ \mathscr{S}_+$};
\draw (-4.5,0.8) node [right] {$ \mathscr{S} $};
\draw (-5,-1.8) node [right] {$ \mathscr{S}_-$};

\draw (-5,1) -- (3.5,1) -- cycle;
\draw (-5,-1) -- (3.5,-1) -- cycle;
\draw  [dotted, thick, color=gray] (-4.7,-1)--(-4.7,1)-- cycle;
\draw [->,>=latex](-4.7,-0.9) --(-4.7,-1);
\draw [->,>=latex](-4.7,0.9) --(-4.7,1);
\draw (-5,0) node[left]{$\frac{2\dz}{|f|}$};
\draw[thick] plot[smooth] coordinates
{(-4.5,0.4) (-4.2,0.34) (-4,0.4)(-3.8,0.3)(-3.5,0.4)(-3,0.35)(-2.5,0.3)
(-2.3,0.4)(-2.2,0.43)(-2,0.44)(-1.5,0.35)(-1,0.3)(-0.5,0.4)
(0,0.43)(0.5,0.3)(1,0.35)(1.5,0.4)(2,0.38)(2.5,0.43)(3,0.44)};

\draw [thick]plot[smooth] coordinates
{(-4.5,-0.41) (-4.2,-0.30) (-4,-0.39)(-3.8,-0.4)(-3.5,-0.35)(-3,-0.4)(-2.5,-0.45)
(-2.3,-0.5)(-2.2,-0.51)(-2,-0.3)(-1.5,-0.35)(-1,-0.37)(-0.5,-0.41)
(0,-0.37)(0.5,-0.35)(1,-0.38)(1.5,-0.45)(2,-0.40)(2.5,-0.41)(3,-0.37)};
\fill (-0.7,-1) circle (1pt);
\draw (-0.7,-1.1) node [below] {$x_\dz$};
\draw[->,>=latex, thick, color=blue](-0.7,-1) --(-0.7,2);
\draw (-0.7,2) node[right]{$f=e_8-e$};
\draw[->,>=latex,thick,color=blue](-0.7,-1) --(1.6,0.8);
\draw (1.6,0.8) node[left]{$q_{e'}$};

\draw  [thick] (-0.7,-1) --(2.4,0.7);
\draw  [thick] (-0.7,-1)--+(0:0.6) arc (0:29.2:0.6cm)--cycle;
\draw (-0.15,-0.8) node[right]{$ \eta$};

\draw [thick,color=blue] (-0.7,-1)--+(38.8:0.4) arc (38.8 :90:0.4cm)--cycle;
\draw (-0.7,-0.55) node[right]{$\alpha$};

\draw [] (-0.7,-1)--+(0:3.5) arc (0 :180:3.5cm)--cycle;

\draw (-3,0) node[right]{$U$};
\fill (2.49,0.44) circle (1pt);
\draw [->,>=latex,thick, color=red](-0.7,-1)--(2.62,0.17);
\draw (2.62,0.18) node [right]{$y$};

\draw (2.5,-1) node [below]{$1$};
\draw (-0.7,-1.7) node [below] {${\rm Figure \ 5}$};
\end{tikzpicture}
\end{center}

From \eqref{xeq4.29}, we have that
for any  $y\in U\cap \partial B(x_\dz,1)$,
\begin{equation}\label{xeq4.33}
u(y)\le u(z_3)+e\cdot (x_\dz-z_3)+e\cdot (y-x_\dz)+4\dz.
\end{equation}
Let $p_y\in H^{-1}(H(e)) $ be such that
$$C_{H(e)}^H(y-x_\dz)=p_y\cdot (y-x_\dz).$$
Then by Lemma \ref{la-mu}, there exists $t_y>0$
such that $t_y( y-x_\dz) \in\partial H(p_y)$.
{Since $\measuredangle (y-x_\dz, f)>\frac\pi2-\frac{C\dz}{|f|}$
for some constant $C\ge 2$ and $|f|\ge\eta$,
we know that if $\dz<\frac{ \eta^2}{C}$, then
$\measuredangle (y-x_\dz, f)>\frac\pi2-\eta $.}
Since we assume $\alpha<\frac{\pi}2-\eta$, we conclude
that $|p_y-e|\ge \eta$.
If  we further choose $\dz<\frac14\phi_{2R}(\eta)$,
then by Proposition \ref{LEM4.5}, we have that
$$  (p_y-e)\cdot  (y-x_\dz)\ge 4\dz.$$
Thus
$$C_{H(e)}^H(y-x_\dz)=p_y\cdot (y-x_\dz)\ge e\cdot (y-x_\dz)+4\dz.$$
This and  \eqref{xeq4.33} yield that
    \begin{equation}\label{xeq4.34}
u(y)\le u(z_3)+e\cdot (x_\dz-z_3)+C_{H(e)}^H(y-x_\dz),
\ \forall y\in U\cap \partial B(x_\delta, 1).
\end{equation}

  It follows from \eqref{xeq4.30} and \eqref{xeq4.34}, and Lemma
  \ref{LEM7.13} that for sufficiently small $\dz>0$, we have
 \begin{equation}\label{eq4.37}
 u(y)\le u(z_3)+e\cdot (x_\dz-z_3)+ C^H_{H(e)}(y-x_\dz),
 \ \forall y\in U\cap B(x_\dz,1).
   \end{equation}
From (A2), we see that for any $q_e\in\partial H(e)$,
$$\{x_\dz+t q_e:t\ge 0\}\cap (U\cap B(x_\dz,1))\neq \emptyset.$$
Hence there exists $t_0>0$ such
that $x_\dz+t_0 q_e\in U\cap B(x_\dz,1)$. Hence
by \eqref{eq4.37} and Lemma \ref{la-mu} we have that
 \begin{align*}
 u(x_\delta+t_0 q_e)&\le u(z_3)+e\cdot (x_\dz-z_3)+C^H_{H(e)}(t_0q_e)\\
 &= u(z_3)+e\cdot (x_\dz-z_3)+t_0 e\cdot q_e\\
 &=u(z_3)+e\cdot (x_\delta+t_0q_e-z_3),
 \end{align*}
 which is impossible, since
 $x_\delta+t_0q_e\in U$ implies
 that $u(x_\dz+t_0q_e)>u(z_3)+e\cdot (y_0-z_3)$.
Therefore we must have $\alpha>\frac{\pi}2-\eta$,
which completes the proof of  Lemma \ref{xangle}.
\end{proof}

Finally, we can give a proof of Lemma \ref{xangler}.

\begin{proof}[Proofs of Lemma  \ref{xangler}.]

Define $v_r(x)=\frac{u(rx)}{r}$ for $x\in B(0,8) $. Then
$
e_{0,8}\in \mathscr D(v_r)(0,8,\dz)$.  Note that
for $\dz<\dz_0$, it holds  $\frac12\le H(e_{0,8})\le 4$.
Since Lemma \ref{xcon-com} implies that
$\{y\in\rr^2:u(y)>u(z_3)+e\cdot (y-z_3)\}\cap B(0,6r)$ has two
connected components that intersect $B(0,r)$,
we see that
the set
$$\{y\in\rr^2:v_r(y)>v_r(z^r_3)+e\cdot (y-z^r_3)\}\cap B(0,6)$$
has two connected components that intersect $B(0,1)$,
here $z^r_3=\frac{z_3}{r}$. Now Lemma  \ref{xangler}
can be proven by applying that of Lemma  \ref{xangle},
with $u$ replaced by $v_r$. We omit the detail.
\end{proof}


\section{Proofs  of Theorem \ref{c2} and Theorem \ref{lla0}:
({\bf A}) $\Rightarrow$ ({\bf C})  and  ({\bf A}) $\Rightarrow$  ({\bf D})}

In this section, we will utilize Theorem \ref{xprop} to complete the
proof of Theorem \ref{lla0}. We begin with

 \begin{thm}\label{c2} For $n=2$, assume that $H$ satisfies both
 \eqref{assh0}   and ({\bf A}) of Theorem \ref{lla}.

  \begin{enumerate}
  \item[(i)] For any domain $\Omega\subset\rn$,
   if $u\in AM_H\left(\Omega\right)$  then $u\in C^1\left(\Omega\right)$.

 \item[(ii)] For any $k>0$, there exists a continuous, monotone increasing
function $\rho_k$,   with $\rho_k\left(0\right)=0$, such that for
any $z\in\rr^2$ and $r>0$, if $v\in AM_H(B\left(x,2r\right)$
satisfies $\|H(Du)\|_{L^\fz(B(x,2r))}\le k$, then
  \begin{equation}\label{rho}\sup_{ x,y\in B\left(z,s\right)}
  |Dv(x)-Dv(y)|\le \rho_k\big(\frac{s}{r} \big). \ \forall s<r.
  \end{equation}
  \end{enumerate}
    \end{thm}

Assume Theorem \ref{c2} for the moment,
we can prove   ({\bf A}) $\Rightarrow$ ({\bf C}) and ({\bf D})  in Theorem \ref{lla0} as below.

 \begin{proof}[Proofs of  ({\bf A})$\Rightarrow$({\bf C})  and   ({\bf D}) in Theorem \ref{lla0}.]
 As pointed in the introduction, it suffices to prove Theorem \ref{lla0} under the assumption  \eqref{assh0}.
({\bf A})$\Rightarrow$({\bf C}) follows directly from Theorem \ref{c2}.
To show ({\bf A})$\Rightarrow$({\bf D}),    let $u\in AM_H(\rr^2)$
satisfies a linear growth at the infinity. Then by Corollary \ref{llg} we have
that  $k:=\|H(Du)\|_{L^\fz(\rr^2)}<\fz$. By Theorem \ref{c2}, $u\in C^1(\rr^2)$.
For any $R >0$, define $u_R(x)=u(Rx)/R$ for all $x\in\rn$.
 Then $u_R\in AM_H(\rr^2)\cap C^1(\rr^2)$ and $\|H(Du_R)\|_{L^\fz(\rr^2)}=k$.
Applying Theorem \ref{c2} again, we see that
 $$ |Du(x)-Du(0)|=  \limsup_{R\to\fz}\Big|Du_R\big(\frac xR \big)-Du(0)\Big|\le \limsup_{R\to\fz}\rho_k\big(\frac {|x|}R \big)=0,$$
 holds for all $x\in\rn .$ This implies that $Du \equiv Du(0)$ and hence $u(x)=u(0)+Du(0)\cdot x$ for all $x\in\rn$.
\end{proof}

Before we prove Theorem \ref{c2}, we need
\begin{lem}\label{xprop1} For $n=2$, assume
$H$ satisfies \eqref{assh0}  and ({\bf A}) of Theorem \ref{lla}.
For each  $\ez >0$ and each vector  $e\in\rr^2$,
there exist  $ \dz_\ast(H,\ez,e  ) >0$  such that for any
$0<\dz<\dz_\ast$, $r>0$, $x\in\rr^2$
 and $u \in AM_H(B(x ,r))$, we have
 \begin{equation}\label{xlin3}
 \max_{e'\in \mathscr D u(x)}| e  -e' |\le \ez\quad
 \mbox{whenever
  $e  \in \mathscr D(u )(x;r;\dz)$.}
\end{equation}
\end{lem}

\begin{proof}  Given any $e  \in \mathscr D(u )(x;r;\dz )$,
if $e=0$, then  Lemma \ref{xxlem} implies that there exists $C_0>0$ such that
$$|u (y)-u(x)|\le r \dz   = \dz|x-y|  \le C^H_{C_0\dz}(y-x)\quad\forall y\in\partial B(x,r ).$$
Hence by $u\in CC_H(B(0,r))$,  we have that
$$|u (y)-u(x)|\le C^H_{C_0\dz}(y-x), \forall y\in B(x,r).$$
Therefore $Su(x)\le \le C_0\dz$, and for any $e'\in \mathscr D u(x)$
$H(e')= Su(x)\le C_0\dz$ and hence
$|e'|\le  C \dz$, where  $C>0$ depends only on $H$.
In particular, it holds that
$$\max_{\{H(e')=H(e)\}}|e'|\le  \ez,$$
if we choose $\delta>0$ such that $C\dz\le \ez$. This yields \eqref{xlin3}.

If $e\ne0$, let $\wz H(p)= \frac1{H(e)}H(p)$.
Define $v_{x, r}(y)=8r u(x+ \frac{y}{8r})$ for $y\in B(0,8)$.
Then $\wz H(e)=1$, $v_r\in AM_H(B(0,8))$,
and $e\in \mathscr D v_r (0;8;\dz)$.
If $8\dz<\dz_\ast (\wz H,e,\ez)$, then by Theorem \ref{xprop}
we have that
 $$\max_{e'\in \mathscr D v_r(0)}| e  -e' |\le \ez.$$
This, combined with the fact that $\mathscr D v_r(0)=\mathscr D u(x)$, implies
\eqref{xlin3} as desired.
\end{proof}

Now we apply Theorem \ref{xprop}, Theorem \ref{lla}, and Lemma \ref{xprop1}
to give a proof of  Theorem \ref{c2}.
\begin{proof}[Proof of Theorem \ref{c2}.]
   We first prove the everywhere differentiability of  $u$  and $H(Du)=Su$ everywhere. For any $x_0\in\Omega$, it suffices to show that $\mathscr Du(x_0)$ is a singleton. For this, let $e\in  \mathscr Du(x_0)$. Then
  there exists a sequence $r_j\to 0$  such that
  \begin{equation}\label{c21}
 \lim_{ j\to\fz}\sup_{y\in B( 0,1)}\Big|\frac1{r_j}\big(u(r_jy+x_0) -u(x_0)\big)-e\cdot y \Big|=0.
 \end{equation}
Thus for each $\dz>0$, there exists $j_\dz $ such that
$e\in\mathscr Du(x_0;r_{j_\dz};\dz )$.
For any $\ez>0$, let $\dz_\ast(H,\ez,e)$ be given by Lemma \ref{xprop1}.
Then, by applying Lemma \ref{xprop1} to $\dz<\dz_\ast$ and $j\ge j_\dz$, we obtain
$$|e-e'|\le \ez,\quad\forall\ e'\in\mathscr Du(x_0).$$
Sending $\ez\to0 $, we conclude that
$e=e'$ for all $e'\in\mathscr Du(x_0)$.
This implies that $\mathscr Du(x_0)$ is a set of single point.

 Next we prove \eqref{rho} by a contradiction argument.
 Suppose that \eqref{rho} were false. By simple translation and dilation,
 Then there would exist
 $k_0>0,\ez_0 > 0$, a sequence $r_j\to 0$, a sequence ${u_j}\subset AM_H(B(0,1))$ with $u_j(0)=0$, and $y_i\in B(0,r_j)$ such that
 \begin{equation}\label{c24}
\|H(Du_j)\|_{L^\fz(B(0,1))}\le k_0,\quad |Du_j(y_j)-Du_j(0)|\ge \ez_0.
\end{equation}
It is readily seen that $\| Du_j \|_{L^\fz(B(0,1))}\le C(H,k)$. Hence
we may assume that $u_\fz\in AM_H(B(0,2))$, with $u_\fz(0)=0$,
such that after passing to a subsequence, $u_j\to u_\fz$ uniformly in $B(0,1)$.
By (i), $u$ is differentiable in $B(0,1)$ so that for any $\dz> 0$, there exists
$r_0 >0$  such that we have
\begin{equation}\label{c25}
\sup_{x\in B(0,r_0)}|u_\fz(x) -Du_\fz(0)\cdot x | \le  {\dz r_0}.
\end{equation}
Therefore there exists a sufficiently large $j_\dz>0$ such that
\begin{equation}\label{c26}
\sup_{x\in B(0,r_0)}|u_j(x) -Du_\fz(0)\cdot x | \le 2\dz r_0,\quad \forall j\ge j_\dz,
\end{equation}
 and hence
\begin{equation}\label{c28}
\sup_{x\in B(0,r_0/2)}|u_j(x+y_j)-u_j(y_j)-Du_\fz(0)\cdot x | \le 4\dz r_0, \quad \forall j\ge j_\dz.
\end{equation}
This implies that $Du_\fz(0)\in \mathscr Du_j(0,r_0,\dz)\cap \mathscr Du_j(y_j,\frac{r_0}{2},8\dz)$.
Let $\dz_\ast(H,\frac{\ez_0}{4},Du_\fz(0))>0$ be given by \ref{xprop1}.
Then, by applying Lemma \ref{xprop1} to $\dz<\frac16 \dz_\ast$ and $j\ge j_\dz$, we obtain that
$$
|Du_\fz(0)-Du_j(y_j)|\le \frac{\ez_0}{4},
\ \ \  \mbox{and}\ \ \ |Du_\fz(0)-Du_j(0)|\le \frac{\ez_0}{4}.
$$
This implies
$$|Du_j(y_j)-Du_j(0)|\le  \frac{\ez_0}{2}, \ \forall j\ge j_\dz,$$
which contradicts to \eqref{c24}. Hence the proof of  Theorem \ref{c2}
is complete.
\end{proof}


\medskip
 \noindent {\bf Acknowledgment}. The authors
 would like to thank  Professor Yifeng Yu for several valuable discussions
 of this paper. Wang is partially supported by NSF DMS 1764417.
 Zhou is partially supported by National Natural Science Foundation of China (No. 11522102, 11871088).

%
%
%

%
%
%

\end{document}